\documentclass[10pt,reqno]{amsart}
\usepackage[margin=1.45in]{geometry}
\usepackage{amsmath,amsfonts,amsthm,amssymb,enumerate,graphicx}
\usepackage{tikz,tikz-cd,url,mathtools}
\usepackage{hyperref}

\let\div\relax
\DeclareMathOperator\div{div}

\newcommand\dd[0]{\partial}
\newcommand\grad[0]{\nabla}
\newcommand{\eps}{\frac{1}{100}}
\newcommand\eq[1]{\begin{align}{#1}\end{align}}
\newcommand\eqn[1]{\begin{align*}{#1}\end{align*}}

\theoremstyle{plain}
\newtheorem{thm}{Theorem}[section]
\newtheorem{lem}[thm]{Lemma}

\newtheorem{prop}[thm]{Proposition}

\theoremstyle{definition}
\newtheorem{define}[thm]{Definition}

\theoremstyle{remark}
\newtheorem{remark}[thm]{Remark}

\numberwithin{equation}{section}

\title[Non-uniqueness for the dyadic Navier--Stokes]{Non-uniqueness in the Leray--Hopf class for a dyadic Navier--Stokes model}

\author{Stan Palasek}
\address{School of Mathematics, Institute for Advanced Study, 1 Einstein Dr., Princeton, NJ, 08540, USA}
\email{palasek@ias.edu}

\begin{document}

\begin{abstract}
    The uniqueness of Leray--Hopf solutions to the incompressible Navier--Stokes equations remains a significant open question in fluid mechanics. This paper proposes a potential mechanism for non-uniqueness, illustrated in a natural dyadic shell model. We show that, for the Obukhov model with $d>2$, there exist initial data at the critical regularity that give rise to two distinct Leray--Hopf solutions. These solutions exhibit an approximately discretely self-similar structure, with non-uniqueness resulting from a partial breaking of the scaling symmetry. The fundamental observation is that, in a certain scenario, the dynamics reduce to a sequence of weakly coupled finite-dimensional systems. Moreover, the predominant nonlinear interactions are identical to those arising in convex integration, suggesting the possibility of a similar construction in the full PDE setting.
\end{abstract}

\maketitle

\section{Introduction}

Consider the initial value problem for the incompressible Navier--Stokes equations,
\begin{equation}\begin{aligned}\label{navierstokes}
\dd_tu-\nu\Delta u+\mathbb P\div u\otimes u=0\\
u(0,x)=u^0(x)
\end{aligned}\end{equation}
where $u:\Omega\to\mathbb R^d$ is an unknown vector field, $\nu>0$ is the viscosity parameter, $\mathbb P$ is the Leray projection onto the space of divergence-free vector fields, and $u^0$ is some divergence-free initial data. In this paper one should consider the spatial domain to be $\Omega=\mathbb R^d$ or $\mathbb T^d$, with $d\geq3$.

A central challenge in the field of mathematical fluid mechanics is to understand whether, for any reasonable class of solutions, the system \eqref{navierstokes} is well-posed globally in time. This question can be approached from two angles: classical/mild solutions, for which uniqueness and regularity are well-established but global existence is unknown \cite{leray,fujitakato}; and weak (Leray--Hopf) solutions, for which global existence is known but uniqueness and regularity are open. In this paper we are concerned with the second approach, where numerical evidence \cite{jiasverak,guillodsverak} has recently emerged to suggest that uniqueness fails for data at the critical regularity. Despite a great deal of interesting partial results (see Section~\ref{navierstokespdenonuniquenesssection} for a survey), a rigorous proof of non-uniqueness in the Leray--Hopf class in the absence of an external force remains out of reach.

The purpose of this paper is to propose a new non-uniqueness mechanism for Leray--Hopf solutions that might be amenable to a rigorous construction in the setting of the Navier--Stokes equations; see Section~\ref{PDEpropsects} for some considerations in this direction. To make the idea concrete and prove a rigorous result, we restrict ourselves to the dyadic setting in which the PDE is replaced by an infinite-dimensional ODE system representing how the fluid's energy flows through Fourier space. Even for this simplified problem, non-uniqueness of unforced Leray--Hopf solutions for Navier--Stokes-type equations has remained open.

Our main theorem pertains to the following model of the Navier--Stokes equations: for an unknown $u:[0,\infty)\to\mathbb R^I$,
\begin{equation}\label{system}
    \begin{aligned}
        \dd_tu_k+\nu N_k^2u_k+B_k[u,u]&=0\\u_k(0)&=u_k^0
        \end{aligned}\quad\quad\quad \forall k\in I
\end{equation}
where $N_k=\lambda^kN_0$ ($\lambda>1$) is a sequence of frequency scales and the nonlinearity is
\eqn{
B_k[u,u]=-N_{k-1}^{\alpha}u_{k-1}u_k+N_{k}^\alpha u_{k+1}^2.
}
By rescaling, we normalize the viscosity $\nu=1$. Unless otherwise noted, we take the index set to be $I=\mathbb N$ which models the Navier--Stokes on $\Omega=\mathbb T^d$. In this case one closes the system by fixing $u_{-1}=0$. At times we will also refer to \eqref{system} with $I=\mathbb Z$ which models the Navier--Stokes on $\mathbb R^d$. An additive group structure on $I$ is necessary to discuss exact scale invariance.\footnote{Transferring the non-uniqueness scenario from $I=\mathbb Z$ to $\mathbb N$ is analogous to the idea of truncating a self-similar solution of a PDE which is a standard idea (e.g., \cite{elgindiII}). In this work the truncation is possible not by stability (because the non-uniqueness is not stable), but rather because the different frequency shells are weakly coupled together, so the self-similarity can be interrupted at the low modes without difficulty.}

The nonlinearity $B_k$ was first formulated by Obukhov~\cite{obukhov} as a model of turbulence in the atmosphere. Alternately, in the language of Tao's quadratic circuits~\cite{taoaveraged}, it represents two amplifier gates chained together. It is one of the two\footnote{The other is the Desnyansky--Novikov~\cite{dn}, also known as Katz--Pavlovi\'c~\cite{katzpavlovicfinitetimeblowup} model. It appears likely that the mechanism described here can be used to prove a similar theorem for that model; however, for the reasons described in Section~\ref{PDEpropsects}, the result may not be as transferable to the Navier--Stokes equations.} fundamental dyadic nonlinearities involving nearest neighbor interactions and we will argue that it is the more natural one for searching for non-uniqueness. Some motivation comes from convex integration schemes for fluid equations, where the two most important interactions correspond to the two terms in $B_k$ (see Section~\ref{PDEpropsects} for details).

The choice of the nonlinearity and meaning of the parameter $\alpha$ is important and subtle, but we defer the full discussion to Section~\ref{dyadicmodelsection}. For now let us just remark that in spatial dimension $d$, the range of physical parameters is $\alpha\in[1,\frac d2+1]$, with the most popular choices being $\alpha=1$ (homogeneous turbulence, any dimension) and $\alpha=\frac52$ ($O(1)$ eddies per scale in dimension three). Which value occurs in turbulence in nature is not a settled question in the physics literature~\cite{iyer2020scaling}.

The nonlinearity $B$ shares two fundamental properties with that of the true Navier--Stokes equations \eqref{navierstokes}: first, it formally conserves energy:
\eqn{
\frac d{dt}\sum_{k\in I}\frac{u_k^2}2+\sum_{k\in I}N_k^2u_k^2=0
}
which can be (formally) derived from \eqref{system} by a telescoping series. Second, with the index set $I=\mathbb Z$, it possesses the group\footnote{This is a discrete subgroup of the familiar 1-parameter group of scaling symmetries for the Navier--Stokes on $\mathbb R^d$ when one associates the left shift operator with the transformation $f(x)\mapsto f(\lambda^{-1}x)$.} of scaling symmetries
\eq{\label{scalinginvariance}
u(t)\mapsto\lambda^{(\alpha-2)n}L^nu(\lambda^{-2n}t),\quad n\in\mathbb Z
}
where $L$ is the left shift operator $(\ldots,a_{-1},a_0,a_1,\ldots)\mapsto(\ldots,a_0,a_1,a_2,\ldots)$. From here it is easy to see that $L^\infty([0,T];H^{\alpha-2})$ is a critical space (c.f.\ the case $\alpha=\frac d2+1$ which reduces to the familiar critical space $H^{\frac d2-1}$ for the Navier--Stokes equations). Moreover, observe that the energy is supercritical---and thus we expect various forms of ill-posedness---when $\alpha>2$.

Let us emphasize that $u_k$ represents not the velocity but more precisely the $L^2$ mass of the velocity in the $k$th frequency shell. This explains the factor $\lambda^{\alpha-2}$ in \eqref{scalinginvariance} instead of the more familiar $\lambda^{-1}$.

In order to state the main theorem, we need the following definition.
\begin{define}
For any initial data $u_0\in L^2$, we say that a solution $u(t)$ of \eqref{system} on $[0,T]$ is \emph{Leray--Hopf} if it belongs to the space\footnote{See Section~\ref{notationsection} for the definition of $H^s$ and $L^2$ for sequences.} $L_t^\infty L^2\cap L_t^2H^1$ and obeys the energy inequality
\eq{\label{energyinequality}
\sum_{k\in I}\frac{u_k^2(t)}2+\int_0^t\sum_{k\in I}N_k^2u_k^2(s)ds\leq\sum_{k\in I}\frac{(u_k^0)^2}2
}
for all $t\in[0,T]$.
\end{define}

A well-known argument based on taking a weak limit of solutions of the Galerkin-truncated system shows that for every $L^2$ initial data, there exists a global-in-time Leray--Hopf solution for \eqref{system}. On the other hand, to our knowledge, the uniqueness of these solutions (without forcing) for \eqref{system} or any other dyadic model of the Navier--Stokes has been an open question up to this point.

Let us now state the main theorem.

\begin{thm}\label{maintheorem}
    For any $\alpha\in(2,4)$ and all sufficiently large $\lambda$, there exists initial data $u^0\in\cap_{s<\alpha-2}H^s$ that gives rise to two distinct Leray--Hopf solutions of \eqref{system}.
\end{thm}

\begin{remark}
    The lower bound on $\alpha$ is expected to be sharp because if $\alpha<2$, the energy is subcritical and the system is locally well-posed by an elementary application Picard's method. The critical case $\alpha=2$ is more subtle and we do not attempt to address it here; however it is reasonable to expect that there is uniqueness in analogy with the situation for the full Navier--Stokes.
    
    The condition $\alpha<4$ is technical and can certainly be removed by a more refined analysis (see Footnote~\ref{alpha4footnote} on p.~\pageref{alpha4footnote}). Note that the range $2<\alpha<4$ already encompasses the cases of greatest interest including the sharp range of parameters for $d=3$ and $4$. 
\end{remark}

\begin{remark}
    The solutions $u,\,v$ are non-negative which is natural in view of the fact that they model the energy spectrum. Moreover, for every $t>0$, they decay exponentially in frequency:
    \eqn{
    u_k(t),\,v_k(t)\lesssim_\lambda N_k^{-\alpha+2} e^{-N_k^2t}
    }
    This corresponds to solutions of the Navier--Stokes which are spatially smooth for positive times. This is a feature in common with the conjectured solutions of Jia and \v Sver\'ak \cite{jiasverakinventiones,jiasverak} as well as the forced solutions of Albritton, Bru\'e, and Colombo \cite{abc}.
\end{remark}

\begin{remark}
    Setting the distance between the scales $\lambda$ large is physically meaningful in the context of the Obukhov model; see Section~\ref{PDEpropsects}.
\end{remark}

\begin{remark}
    For large choices of the parameter $\lambda$, the data becomes nearly discretely self-similar under rescalings by $\lambda^n$, $n\in\mathbb Z$. For positive times, the symmetry is broken and the solutions are only invariant under rescaling by $\lambda^{n}$, $n\in2\mathbb Z$. $u$ and $v$ represent the two possible ways the symmetry might break. (See Section~\ref{nonuniquenessandsymmetrysection} for more details.) There is an intriguing analogy with the numerical simulations of Guillod and \v Sver\'ak~\cite{guillodsverak} who show how non-uniqueness could emerge from breaking a reflection symmetry. We remark that there has also been work on non-uniqueness and symmetry breaking in the sense of a 3D flow emerging from 2D initial data \cite{bardossymmetrybreaking,wiedemanninviscidsymmetry,barkerprangejinsymmetrybreaking}.
\end{remark}

\begin{remark}\label{criticalityremark}
    In addition to lying in every supercritical $H^s$ space, the data satisfies $u_k^0\sim_\lambda N_k^{-\alpha+2}$ which puts it in the critical space $X^{\alpha-2}$ (see Section~\ref{notationsection}). Analogizing to the full Navier--Stokes system, one can compute that this corresponds to the velocity field belonging to the Besov space $B_{p,\infty}^{-1+2(\alpha-1)/p}$ for all $p\in[1,\infty]$. In the case of maximum intermittency (for instance, self-similar solutions), one has $\alpha=\frac d2+1$ and the space becomes $ B_{p,\infty}^{-1+d/p}$ which is a critical space lying just beyond the local well-posedness theory.

    We briefly present the heuristics that suggest the velocity would lie in this space. The reader unfamiliar with the numerology of dyadic models might first consult Section~\ref{dyadicmodelsection}. With $v$ a vector field corresponding to the dyadic solution $u$, we have
    \eqn{
    \|v\|_{ B_{p,\infty}^{-1+2(\alpha-1)/p}}\sim_\lambda\sup_{k\geq0}(N_k^{-1+\frac{2}p(\alpha-1)}\|P_{\sim_{N_k}}v\|_{L^p})
    }
    where $P_{\sim N_k}$ represents projection to the frequency shell around $N_k$. We may conceptualize $P_{\sim N_k}v$ as a sum of disjoint wavelets $\psi_1+\cdots+\psi_{m_k}$ with amplitude\footnote{The Obukhov model only allows one amplitude per frequency shell. Of course it would be more realistic if all the wavelet coefficients evolved separately.} $\sim N_k^{\alpha-1} u_k$ (the solution of \eqref{system}). Each wavelet is supported on a spatial set of volume $\sim~ N_k^{-d}$. Recall that the parameter $\alpha$ gives that there are $\sim m_k=N_k^{d-2(\alpha-1)}$ eddies per frequency shell. Thus we can estimate
    \eqn{
    \|P_{\sim N_k}v\|_{L^p}\lesssim |u_k|N_k^{-\frac2p(\alpha-1)}\lesssim N_k^{1-\frac2p(\alpha-1)}
    }
    using that $|u_k|\lesssim N_k^{-\alpha+2}$. From this we conclude.
\end{remark}

\begin{remark}\label{instabilityremark}
    The scenario is unstable in the sense that we do not expect it to persist for any open set of data about $u^0$. There is a destabilizing ``drift''-type term fighting against the stabilization from the dissipation, and unfortunately the former dominates on the relevant time scale. See Section~\ref{2dheuristicssection} for more details.
\end{remark}

\subsection{Previous results for the Navier--Stokes equations}\label{navierstokespdenonuniquenesssection}

There has been a great deal of work in the PDE setting toward exhibiting non-uniqueness in the Leray--Hopf class. Note that we say that a vector field $u$ is a Leray--Hopf solution of \eqref{navierstokes} if it solves the PDE in a distributional sense and obeys an energy inequality analogous to \eqref{energyinequality}. It is a classical theorem of Leray~\cite{leray} that these solutions exist for any divergence-free $L^2$ data. Let us also recall that \eqref{navierstokes} has a continuous scaling symmetry analogous to \eqref{scalinginvariance} that defines a notion of criticality and motivates introducing the forward self-similar ansatz
\eqn{
u(t,x)=\frac1{\sqrt t}U(\frac x{\sqrt t},\log t).
}
This is the starting point for the program of Jia and \v Sver\'ak \cite{jiasverakinventiones,jiasverak} for constructing non-unique Leray--Hopf solutions from self-similar data. If one can find a profile $U$ that has an unstable eigenvalue for the dynamics in self-similar variables, this can be translated into non-uniqueness in the original variables. While there is compelling numerical evidence that such an instability exists~\cite{guillodsverak}, there is not a clear path to a rigorous proof.

One can relax the problem by introducing an external force, the result being that the profile $U$ does not have to solve any PDE. This was accomplished by Albritton, Bru\'e, and Colombo~\cite{abc} who constructed an unstable vortex ring based on an unstable 2D vortex constructed by Vishik~\cite{vishik1,vishik2}. In the dyadic setting, there is also an example of non-unique Leray--Hopf solutions in the presence of an external force due to Filonov and Khodunov~\cite{filonovkhodunov}. At this point there are many works, including by the author, obtaining strong non-uniqueness theorems by including an external force~\cite{vishik1,vishik2,abc,bulut2023convex,bulut2023non,dai2023non,hofmanova2023non}. It is important to mention that in all these examples, the forces are quite singular near the initial time and so do not rule out the possibility of a reasonable well-posedness theory.

Convex integration offers another approach to constructing non-unique solutions of fluid equations, introduced for the Euler equations by De Lellis and Sz\'ekelyhidi \cite{delellisszekelyhididifferentialinclusion,delellisszekelyhidicontinuousdissipative}. Buckmaster and Vicol \cite{buckmastervicol} used this method to construct examples of non-uniqueness for the 3D Navier--Stokes equations in the relatively low regularity class $C_tH_x^\epsilon$. Tao~\cite{taoblognavierstokesconvexintegration} extended these ideas to the space $C_tH^{\frac12-\epsilon}$ for the Navier--Stokes in $\mathbb R^d$, where $d$ is chosen large depending on $\epsilon>0$. Unfortunately $H^\frac12$ appears to be a hard limit for convex integration for PDEs of this type, which is quite far from the energy space $L_t^2H_x^1$. Another route is to weaken the dissipation so that the energy space lies at a lower regularity. In this way De Rosa~\cite{derosalerayhopf} proves non-uniqueness of Leray--Hopf solutions of the hypoviscous Navier--Stokes equations in which the viscosity term $-\Delta$ is replaced by $(-\Delta)^\gamma$, $\gamma<\frac13$. Due to criticality considerations, one does not expect this approach to extend to larger $\gamma$. Finally, let us mention the work of Cheskidov and Luo~\cite{cheskluosharpnonunique} that uses a time intermittent convex integration scheme to demonstrate non-uniqueness in $L_t^pL_x^\infty$, $p<2$. Remarkably, this scale of spaces approaches the critical space $L_t^2L_x^\infty$ on the Prodi--Serrin--Ladyzhenskaya scale. Still, these solutions are very far from having finite energy.

\subsection{Dyadic models}\label{dyadicmodelsection}

Shell models of turbulence have a long history beginning in the physics literature; we refer the reader to the recent survey~\cite{dyadicreview} for a more complete account. Broadly speaking, the idea is to model the flow of energy between various frequency shells which we denote by $N_k$, $k\in\mathbb N$. This yields interesting dynamics even in the simplest case where only nearest-neighbor interactions are considered and $u$ is taken to be scalar valued. There are two basic nonlinearities in this class of models: the Desnyansky--Novikov (DN)~\cite{dn}, also known as Katz--Pavlovi\'c (KP)~\cite{katzpavlovicfinitetimeblowup} nonlinearity
\eqn{
A_k[u,u]=-N_k^\alpha u_{k-1}^2+N_{k+1}^\alpha u_k u_{k+1}
}
and the Obukhov~\cite{obukhov} nonlinearity
\eqn{
B_k[u,u]=N_{k-1}^\alpha u_{k-1}u_k-N_k^\alpha u_{k+1}^2.
}
$A$ and $B$ both enjoy the cancellation property that $(A[u,u],u)_{\ell^2}$ and $(B[u,u],u)_{\ell^2}$ vanish for all $u\in H^{10\alpha}$ (say). This implies energy conservation for ``classical'' solutions. These models are fundamental in the sense that every quadratic nonlinearity in this category that conserves energy is a linear combination of $A$ and $B$~\cite{kiselevzlatosdiscrete}.

There is a clear distinction between the models: $A$ emphasizes the forward cascade of energy and consequently is a popular choice for understanding possible blow-up of fluid equations and the onset of turbulence \cite{katzpavlovicfinitetimeblowup,cfp2007,cheskidovblowup,cfp2010,barbatodyadic2011}; on the other hand, $B$ contains a more robust inverse cascade effect through the term $-N_k^\alpha u_{k+1}^2$. We suggest that $B$ is in fact a more natural model for understanding possible non-uniqueness of Leray--Hopf solutions. Indeed, consider Leray's original construction, in which one constructs distinct solutions $u$ and $v$ as limits of solutions $u_k$ and $v_k$ of Galerkin-truncated systems below frequency $N_k$. Clearly along a subsequence, $u$ and $v$ must differ at some bounded frequency $N_0$; otherwise the difference disappears in the weak limit. Since at the initial time, $u_k$ and $v_k$ only differ at modes $\gtrsim N_k$, the difference between them must cascade downward in frequency. This motivates our effort to capture this cascade through the $-N_k^\alpha u_{k+1}^2$ term in the Obukhov nonlinearity.

One of several interpretations of the dyadic models is as follows. Suppose that at each frequency scale $N_k$, the velocity field is supported on a set of measure $\sim~N_k^{-2(\alpha-1)}$. In order to normalize the energy ($L^2$ norm) one must scale by $N_k^{\alpha-1}$. As a result, one takes $u_k$ such that the amplitude of the projection $P_{\sim N_k}u$ is $\sim N_k^{\alpha-1}u_k$; thus the energy contained in the $k$th frequency shell is $\frac12u_k^2$ and the total energy is $\frac12\sum_{k\geq0}u_k^2$. Based on this amplitude and the original Navier--Stokes nonlinearity one can derive the power $\alpha$ in the formulas for $A$ and $B$.

In this way, the value of $\alpha$ characterizes the intermittency of the turbulence. At one extreme, $\alpha=1$ implies that at every scale, the turbulent behavior occupies a bounded fraction of the volume of the spatial domain---this is the non-intermittent picture in accordance with Kolmogorov's predictions. At the other extreme, by the uncertainty principle, the support of $P_{\sim N_k}u$ must have at least volume $\sim N_k^{-d}$, corresponding $\sim1$ ``eddies'' at the length scale $N_k^{-1}$. This corresponds to $\alpha=\frac d2+1$ and is in agreement with, for instance, self-similar solutions. In this case one expects to find the sharpest counterexamples to the well-posedness theory. In the general case $\alpha\in[1,\frac d2+1]$, the Hausdorff codimension of the turbulence in the domain is given by $2(\alpha-1)$.

\subsection{Heuristics for the binary system}\label{2dheuristicssection}

To understand the scenario that leads to Theorem~\ref{maintheorem}, we must first consider \eqref{system} as a sequence of coupled two-dimensional systems. For instance, we may let $k$ run over the positive odd integers and partition \eqref{system} into subsystems
\begin{equation}\begin{aligned}
\dd_tu_{k-1}&=b_k(t)u_{k-1}-N_{k-1}^\alpha u_k^2\\
\dd_t u_k&=-N_k^2u_k+N_{k-1}^\alpha u_{k-1}u_k+f_k(t)\label{binarysystem}
\end{aligned}\end{equation}
for all $k\in2\mathbb N+1$, where
\eqn{
b_k(t)=N_{k-2}^\alpha u_{k-2}-N_{k-1}^2,\quad f_k(t)=-N_k^\alpha u_{k+1}^2.
}
The coupling between the ODE systems across the various $k\in2\mathbb N+1$ is mediated by the ``drifts'' $b_k(t)$ and ``forces'' $f_k(t)$. We emphasize that \eqref{binarysystem} is nothing more than a reorganization of the original system \eqref{system}.

It is well-understood that as long as the nonlinearity stays reasonably controlled\footnote{It is expected that the dissipative Obukhov model is globally well-posed. However, to our knowledge, this is only known for the inviscid model~\cite{kiselevzlatosdiscrete}, and we do not pursue the finite-time blow-up question here.}, the dissipation term will cause the $j$th mode to decay exponentially on time scale $N_j^{-2}$. The idea of the scenario is to arrange the data in such a way that in the dynamics of \eqref{binarysystem}, $u_{k-1}$ actually decays on the time scale $N_k^{-2}$, much quicker than what is expected from the dissipation term $N_{k-1}^2u_{k-1}$. Thus for every $k$ odd, the pair $(u_{k-1},u_k)$ annihilates on time scale $N_k^{-2}$. This results in an approximate decoupling of the dynamics of \eqref{binarysystem} for different $k$'s. Indeed, on the relevant time scale for $(u_{k-1},u_k)$, the drift $b_k$ is essentially stationary, because $u_{k-2}$ changes on the much longer time scale $N_{k-2}^{-2}$. Similarly, $f_k$ has dissipated away and become essentially negligible for times $t\gg N_{k+2}^{-2}$.

Thus, \eqref{binarysystem} can be approximated by the decoupled system
\begin{equation}\begin{aligned}\label{decoupledsystem}
\dd_tu_{k-1}&\approx b_k(0)u_{k-1}-N_{k-1}^\alpha u_k^2\\
\dd_t u_k&\approx-N_k^2u_k+N_{k-1}^\alpha u_{k-1}u_k
\end{aligned}\end{equation}
for all $k\in2\mathbb N+1$. The advantage of \eqref{decoupledsystem} is that across the different $k\in2\mathbb N+1$, the systems are fully decoupled and each one can be easily understood in isolation. The dynamics are shown in Figure~\ref{binarysystemphaseportraitfigure}. It is easy to see that for this two-dimensional system, $(u_{k-1},u_k)=(0,0)$ is a hyperbolic point from which a stable manifold emerges, part of which is shown in red. We define $\delta_k$ to be the value of $u_{k-1}$ for which $u_k$ is stationary, namely
\eqn{
\delta_k=N_{k-1}^{-\alpha}N_k^2.
}
This scale is essential in our analysis and represents the critical amplitude of $u$.

\begin{figure}
    \centering
    
    \hspace{-.41in}\includegraphics[scale=1.1]{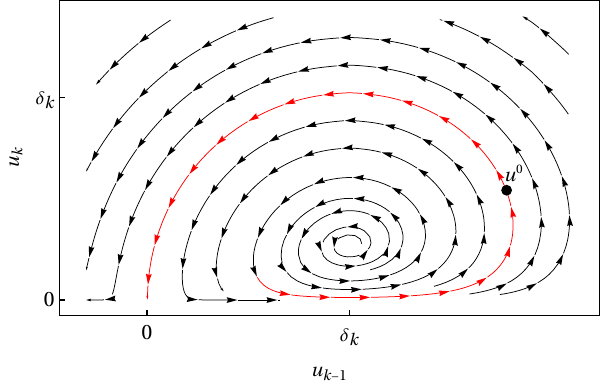}
    \caption{The approximate dynamics of the system for $(u_{k-1},u_k)$, obtained by setting $b_k=b_k(0)$ and $f_k=0$. Part of the stable manifold of $(0,0)$ is highlighted in red and is nearby the semicircle centered at $(\delta_k,0)$ with radius $\delta_k$. In this simplified picture, one should imagine the initial data being chosen on the stable manifold, close to the point $(2\delta_k,0)$, at an angle $\sim\lambda^{-\alpha+2}$ above the horizontal. The result is the $k-1$ and $k$th modes annihilate on the short time scale $N_k^{-2}$.}\label{binarysystemphaseportraitfigure}
\end{figure}

One can see that if the scales are well-separated and $b_k(0)\ll N_k^2$, then this branch of the manifold approximates a semicircle centered at $(\delta_k,0)$ with radius $\delta_k$. In particular, $u_{k-1}$ decays\footnote{More precisely, $u_{k-1}=\frac12\delta_k^{-1}u_k^2+O(u_k^4)$ for large times.} like the square of $u_{k}$ which leads the whole two-dimensional system to decay on time scale $\sim N_{k}^{-2}$ as desired. Conceptually, what is happening is the $(k-1)$st mode is transferring its energy to the $k$th mode, which is in turn dissipating it at a fast rate. By choosing the data carefully, one arranges that this process drains the exact amount needed for $u_{k-1}$ to vanish. Note $b_k$ is positive on the relevant time scale which destabilizes this arrangement; see Remark~\ref{instabilityremark}. In reality, one cannot simply choose initial data freely off a one-dimensional manifold because there are infinitely-many of these systems coupled together, the stable manifold of one depending on the initial data for all the others. This creates a technical challenge that we discuss in Section~\ref{stablemanifoldsection}.

\subsection{Non-uniqueness and breaking the scaling symmetry}\label{nonuniquenessandsymmetrysection}

Here we give an overview of how the dynamics described in Section~\ref{2dheuristicssection} give rise to non-unique Leray--Hopf solutions. For simplicity in this exposition, let us continue to neglect the coupling between the systems \eqref{binarysystem}. So far we have argued that there exist solutions of \eqref{system} such that for every $k\geq1$ odd, the dynamics of $(u_{k-1},u_k)$ approximately follow the trajectory depicted in Figure~\ref{binarysystemphaseportraitfigure} and decay exponentially on time scale $\sim N_k^{-2}$. In fact, almost the same reasoning applies when the $k$'s are instead taken to be even (with just a trivial modification for the zeroth mode). Thus, we obtain a second solution $v$ of \eqref{system} such that $(v_{k-1},v_k)$ obey \eqref{decoupledsystem} for all $k\geq2$ even. The final observation is that the initial data for $u$ and $v$ can be selected on their respective stable manifolds to agree. Assuming the 2D dynamics are fully decoupled, the trajectories become semicircles and the existence of a common datum $u^0$ is clear because one can construct a solution to the algebraic system
\eq{
\left(\frac{u_{k-1}^0}{\delta_k}-1\right)^2+\left(\frac{u_k^0}{\delta_k}\right)^2=1\quad\text{for all }k\geq1\label{dataequation}
}
by a fixed point argument. We illustrate the concept of the non-uniqueness mechanism in Figure~\ref{nonuniquenessconceptfigure}. In practice, due to the coupling between the systems, selection of the data requires a more delicate topological argument. 

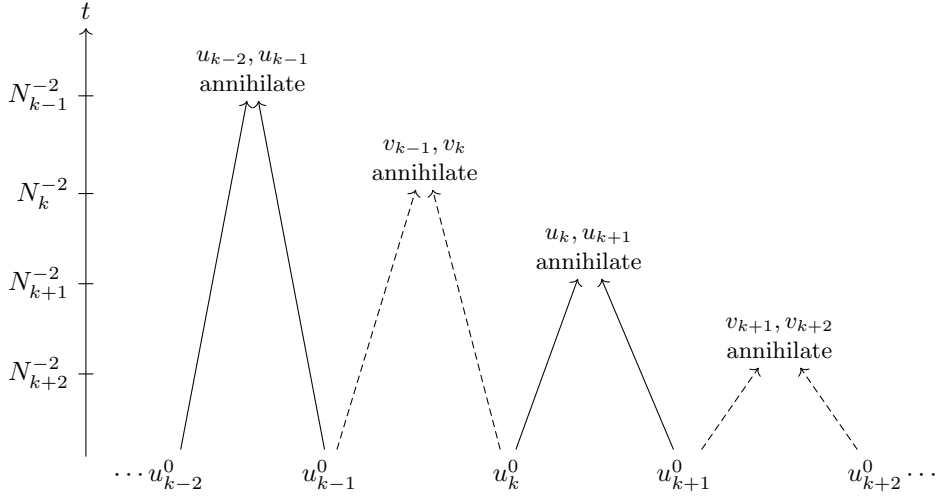
\begin{figure}
    \centering

\begin{minipage}{0.08\textwidth} % Adjust the width as needed
    \begin{tikzpicture}
        \draw[->] (0,2.3) -- (0,8.) node[anchor=south] {$t$};

        \node at (0,1.2) {{\color{white}}};

  % Add ticks and labels
  \draw (.1,3.4) -- (-0.1,3.4) node[anchor=east] {$N_{k+2}^{-2}$};
  \draw (0.1,4.6) -- (-0.1,4.6) node[anchor=east] {$N_{k+1}^{-2}$};
  \draw (0.1,5.8) -- (-0.1,5.8) node[anchor=east] {$N_{k}^{-2}$};
  \draw (0.1,7.1) -- (-0.1,7.1) node[anchor=east] {$N_{k-1}^{-2}$};
    \end{tikzpicture}
\end{minipage}
\hspace{0cm} % Adjust the horizontal space as needed
\begin{minipage}{0.9\textwidth} % Adjust the width as needed
    \begin{tikzcd}[row sep=.4cm, column sep=-.4cm]
&&\parbox{1.5cm}{\small\centering$u_{k-2},u_{k-1}$\\annihilate}&&&&&&&\\
&&&&\parbox{2cm}{\small\centering$v_{k-1},v_{k}$\\annihilate}&&&&&\\
&&&&&&\parbox{2cm}{\small\centering$u_{k},u_{k+1}$\\annihilate}&&&\\
&&&&&&&&\parbox{2cm}{\small\centering$v_{k+1},v_{k+2}$\\annihilate}&\\[2em]
\cdots\hspace{.15cm} &u_{k-2}^0 \arrow[uuuur] &&u_{k-1}^0 \arrow[uuuul] \arrow[uuur, dashed] &&u_k^0 \arrow[uuul, dashed] \arrow[uur] &&u_{k+1}^0 \arrow[uul] \arrow[ur, dashed] &&u_{k+2}^0 \arrow[ul, dashed] &\hspace{.15cm}\cdots
\end{tikzcd}
\end{minipage}

    \caption{The data $u^0$ is such that any adjacent pair of modes $(u_{k-1},u_k)$ can annihilate (by which we mean approximately follow the dynamics in Figure~\ref{binarysystemphaseportraitfigure}) on the time scale $N_{k}^{-2}$. This annihilation takes place in such a way that it is only weakly coupled to the rest of the system. There are two possible configurations which are approximately rescaled and shifted versions of each other: the solid interactions leading to $u$, and the dashed interactions leading to $v$.}\label{nonuniquenessconceptfigure}
\end{figure}

In order to discuss the self-similar properties of the solutions, we temporarily take the index set for the modes to be $I=\mathbb Z$, the natural choice for Navier--Stokes on $\mathbb R^d$. In fact our construction is even simpler in this setting (at the cost of yielding infinite energy solutions).

Applying Taylor's theorem to \eqref{dataequation}, we see that when $\lambda$ is taken large, the data approaches discrete self-similarity in the sense that
\eqn{
u_k^0=\lambda^{-(\alpha-2)k}(1+o_{\lambda\to\infty}(1))U^0
}
where the ``self-similar profile'' is $U^0=u_0^0$. In other words, the data is approximately invariant under the group of scaling symmetries
\eq{
u^0\mapsto\lambda^{(\alpha-2)n}L^nu^0,\quad n\in\mathbb Z\label{datasymmetry}
}
where, once again, $L$ is the left shift operator.

Now consider the solutions $u(t),v(t)$. One can see that as $\lambda$ becomes large and the drift $b_k$ becomes negligible compared to the dissipation rate $N_k^2$, \eqref{binarysystem} reduces further and we have the explicit solution
\eqn{
(u_{k-1},u_k)=\frac{2\delta_k}{1+\lambda^{-2(\alpha-2)}\exp(2N_k^2t)}(1,\lambda^{-(\alpha-2)}e^{N_k^2t})
}
for $k\in2\mathbb N+1$, and similarly for $v$ when $k\in2\mathbb N+2$. It is easy to see that these approximate solutions are invariant under the group of transformations
\eq{\label{solutionsymmetry}
u(t)\mapsto\lambda^{(\alpha-2)n}L^nu(\lambda^{-2n}t), \quad v(t)\mapsto\lambda^{(\alpha-2)n}L^nv(\lambda^{-2n}t),\quad n\in2\mathbb Z,
}
with $u(t)=\lambda^{\alpha-2}Lv(\lambda^{-2}t)$. In summary, the approximate symmetries of $u,v$ (given by \eqref{solutionsymmetry}) form a proper subgroup of the symmetries of the data (given by \eqref{datasymmetry}). Thus we conclude that a breaking of the approximate scaling symmetry takes place at $t=0$.

\subsection{Prospects for the Navier--Stokes equations}\label{PDEpropsects}

Our interest in studying dyadic models is because they might suggest promising ideas for the full PDE system. Thus, some comments about the transferability of Theorem~\ref{maintheorem} to the Navier--Stokes equations are in order. Broadly speaking, one hopes to construct initial data supported in Fourier space on a sequence of shells at frequency $(N_k)_{k=0}^\infty$ such that the energy in each shell qualitatively obeys the dynamics described in Sections ~\ref{2dheuristicssection}--\ref{nonuniquenessandsymmetrysection} (up to some controllable errors from the other nonlinear terms). One can view this proposal as the non-uniqueness analog of Tao's~\cite{taoaveraged} program for proving blow-up of the Navier--Stokes equations based on a dyadic model. One might expect the non-uniqueness problem to be more tractable because the $k$th mode only needs to be controlled on the natural time scale $N_k^{-2}$ before it dissipates away, while the program for blow-up is genuinely global in time.

The significant distinction between the model in \cite{taoaveraged} and \eqref{system} is that while the former contains several different types of quadratic gates linked together in a complicated circuit, the latter is just a chain of amplifier gates (see Section~5.3 in~\cite{taoaveraged}). We contend that amplifier gates are probably the most favorable type for embedding into the true Navier--Stokes system due to their ability to operate across widely separated scales.

First, consider $\lambda$, the separation factor between the scales $N_k$ and $N_{k+1}$. Given that Theorem~\ref{maintheorem} requires $\lambda$ to be large, it is important to distinguish between the DN and Obukhov models (see Section~\ref{dyadicmodelsection}). In the DN model, it would be quite unnatural to take $\lambda$ large\footnote{This reality is reflected in Theorem~3.2 in \cite{taoaveraged} where $\lambda-1\ll1$ is required in order to embed a dyadic model into the averaged Euler equations. It is reasonable to expect that this constraint does not exist for models consisting only of amplifier gates across different frequency shells (such as \eqref{system}). It would follow from Theorem~\ref{maintheorem} that Leray--Hopf solutions are non-unique for an averaged Navier--Stokes equation, but we do not pursue this direction here.} because it contains the nonlinear term $N_k^\alpha u_{k-1}^2$ which, in the PDE setting, corresponds to the nonlinear interaction
\eq{\label{dninteraction}
P_{\sim N_k}\div (P_{\sim N_{k-1}}u)\otimes (P_{\sim N_{k-1}}u).
}
Because multiplication in physical space corresponds to convolution in Fourier space, \eqref{dninteraction} is only non-negligible if $2N_{k-1}\gtrsim N_k$, which fails if $\lambda$ is large. For a similar reason, the other term in the DN nonlinearity is only natural with $\lambda\sim 1$. On the other hand, in the Obukhov model, the terms $N_{k-1}^\alpha u_{k-1}u_k$ and $N_k^\alpha u_{k+1}^2$ correspond respectively to the Navier--Stokes interactions\footnote{If there is any doubt that this is in fact the correct correspondence, one can verify that the Obukhov-type Euler system $\dd_tP_{\sim N_k}u+(P_{\sim N_k}u)\cdot\grad (P_{\sim N_{k-1}}u)+P_{\sim N_k}\div(P_{\sim N_{k+1}}u)\otimes (P_{\sim N_{k+1}}u)+\grad p=0$, $\div u=0$ formally conserves the energy $\frac12\int|u|^2$.}
\eq{\label{nasherror}
(P_{\sim N_k}u)\cdot\grad (P_{\sim N_{k-1}}u)
}
and
\eq{\label{highhighlowinteraction}
P_{\sim N_k}\div(P_{\sim N_{k+1}}u)\otimes (P_{\sim N_{k+1}}u).
}
With these nonlinearities, a vast difference between the frequency scales causes no difficulty, in contrast with $\eqref{dninteraction}$.

This leads us to believe that if one can find Navier--Stokes dynamics such that \eqref{nasherror} and \eqref{highhighlowinteraction} are the strongest nonlinear interactions, then there is hope to embed the dynamics described in Section~\ref{2dheuristicssection}. This is promising because these are the dominant interactions appearing in convex integration schemes\footnote{We refer the unfamiliar reader to the review article~\cite{buckmastervicol2019review}.}; \eqref{nasherror} is what is known as the Nash error, and \eqref{highhighlowinteraction} is the high-high-to-low interaction that is instrumental for canceling out the errors at frequencies $\lesssim N_k$. In other words, there are already known scenarios where \eqref{nasherror}--\eqref{highhighlowinteraction} dominate and the DN-type terms are negligible.

We also remark that, in contrast to nearly all of the non-uniqueness results up to this point, the Laplacian is not treated perturbatively in our construction. Dissipation is an essential part of the non-uniqueness mechanism itself and furthermore has a very substantial role in simplifying the analysis by keeping the various modes active on separated time scales and for just a short time.

\subsection{Notation and definitions}\label{notationsection}

We define the natural numbers $\mathbb N$ to include $0$. Arithmetic operations on sets are defined in the obvious way; for instance we will frequently write $2\mathbb N+1$ and $2\mathbb N+2$ to mean the sets $\{1,3,5,\ldots\}$ and $\{2,4,6,\ldots\}$ respectively.

For real sequences $a$ indexed over $I\subset\mathbb Z$, we define the Sobolev norms
\eqn{
\|a\|_{H^s}=(\sum_{k\in I}\lambda^{2sk}|a_k|^2)^\frac12
}
as well as the weighted $\ell^\infty$ norms
\eqn{
\|a\|_s\coloneqq\sup_{k\in I}\lambda^{sk}|a_k|.
}
Let $H^s(I)$ and $X^s(I)$ be the spaces of real sequences over $I$ such that $\|a\|_{H^s}<\infty$ and $\|a\|_s<\infty$, respectively. In practice the index set will be clear from the contex and we will not write it explicitly. In analogy with the PDE setting, we write $L^2$ as a synonym for $H^0$. $H^s$ in the dyadic setting corresponds quite closely to the usual notion of $H^s(\mathbb T^d)$ (or $\dot H^s(\mathbb R^d)$, depending on the choice of $I$); see the discussion in Section~\ref{dyadicmodelsection}. $X^s$ can be thought of as a whole scale of Besov spaces with interpolation exponent $\infty$; see Remark~\ref{criticalityremark}.

For brevity we introduce the notation $A\lesssim B$ to mean there exists $c>0$ such that $A\leq cB$. $A\sim B$ means that $A\lesssim B\lesssim A$. $O(B)$ stands in for any quantity whose absolute value is controlled by $cB$ for some $c>0$.

The parameters are chosen with the following dependencies. $\alpha\in(2,4)$ should be taken as fixed and dependence on it will never be explicitly noted in the estimates. Notation such as $\lesssim$ and $O()$ may contain constants depending on $\alpha$ but not on $\lambda$. As a result, we may choose $\lambda$ large depending on the constants in $O$ and $\lesssim$. We will use the notation $O_X()$, $\lesssim_X$, and $\sim_X$ to indicate that the implicit constant \emph{is} allowed to depend on the list of parameters $X$.

\subsection{Plan of the paper}

In Section~\ref{stablemanifoldsection}, we construct families of solutions $u$ and $v$ near $t=\infty$ corresponding to the scenario described in Section~\ref{2dheuristicssection}. In Section~\ref{bootstrapsection} we use a bootstrap argument to show that the trajectories can be continued and lie in a certain region until time $t=0$. In Section~\ref{constructthedataprovethetheoremsection}, we use a topological argument to prove that there exist choices of the parameters so that $u$ and $v$ have the same data. In Appendix~\ref{stabilityappendixsection}, we prove several technical facts about stability needed in Section~\ref{constructthedataprovethetheoremsection}.

\subsection*{Acknowledgment}

The author acknowledges support from the NSF under
Grant No. DMS-1926686.

\section{Constructing the stable manifold}\label{stablemanifoldsection}

In this section and the next, in order to analyze the two solutions $u$ and $v$ simultaneously, the notation ``$u$'' will be used to refer to either of the two. We distinguish them based on the parity of the indices that annihilate as described in Section~\ref{2dheuristicssection}. Fixing $\sigma\in\{1,2\}$, we let $k$ run over the set $2\mathbb N+\sigma=\{\sigma,\sigma+2,\ldots\}$ (even when not specified explicitly). These will be the $k$'s for which the modes $u_{k-1},u_k$ interact and drain their energy on time scale $N_k^{-2}$. Thus, we will ultimately construct $u$ and $v$ by appealing to this argument with $\sigma=1$ and $\sigma=2$, respectively.

In the next proposition, we construct the particular branch of the stable manifold of the full infinite-dimensional system at $(0,0,\ldots)$ corresponding (roughly) to the trajectories depicted in Figure~\ref{binarysystemphaseportraitfigure}. To parameterize the desired submanifold, we introduce the parameters $(\tau_k)_{k\in2\mathbb N+\sigma}\in\mathbb R$ which can be thought of as time delays. In Section~\ref{constructthedataprovethetheoremsection}, their values will be set via a topological argument; for now let us just assume
\eq{\label{tauconstraint}
        |N_k^2\tau_k-\log\lambda^{\alpha-2}|\leq5.
        }
        Since the dynamics for the various 2D systems $(u_{k-1},u_k)$ are approximately decoupled, $\tau_k$ approximately has the effect of shifting the $k$th system by time $\tau_k$ relative to the trajectory with asymptotics $u_k\sim\delta_ke^{-N_k^2t}$.

The asymptotic behavior of the solution is constructed by a contraction mapping argument on the interval $[t_0,\infty)$, where
\eqn{
t_0\coloneqq2\alpha N_0^{-2}\log\lambda.
}
The map in question is a Duhamel-type formula that propagates $u_k$ backward-in-time from $t=\infty$ and $u_{k-1}$ forward-in-time from $t=t_0$. Its structure guarantees that any fixed point (with suitable bounds) is not only on the stable manifold but also has the prescribed values
\eqn{
u_k(t_0)=\delta_k\exp(-N_k^2(t_0-\tau_k))\quad\forall k\in2\mathbb N+\sigma.
}

\begin{prop}\label{stablemanifoldprop}
    For any $\sigma\in\{1,2\}$ and $(\tau_k)_{k\in2\mathbb N+\sigma}$ obeying \eqref{tauconstraint}, there exists a solution of \eqref{system} on the time interval $[t_0,\infty)$ such that
    \eq{
    0\leq u_{k-1}\leq3\delta_{k}e^{-2N_{k}^2(t-\tau_k)}\label{manifoldboundsk-1}\\
    |u_k-\delta_ke^{-N_k^2(t-\tau_k)}|\leq\frac12\delta_ke^{-N_k^2(t-\tau_k)}.\label{manifoldboundsk}
    }
    for all $k\in2\mathbb N+\sigma$. Additionally, $u_k(t_0)=\delta_ke^{-N_k^2(t_0-\tau_k)}$ for this range of $k$.
    
    In the case $\sigma=2$, we may freely specify the value of the zeroth mode at the time $t_0$ to be $u_0(t_0)=a$, where $a$ is taken from
    \eq{\label{arange}
    \frac32\delta_1\leq a\leq\frac52\delta_1.
    }
    
\end{prop}

\begin{proof}

Consider the norm
\eqn{
\| u\|_Y=\sup_{k\in2\mathbb N+\sigma}\delta_k^{-1}\|(\frac13e^{2N_k^2(t-\tau_k)}u_{k-1},2e^{N_k^2(t-\tau_k)}u_k)\|_{L^\infty([t_0,\infty))}.
}
Note that when $\sigma=2$, $\|\cdot\|_Y$ does not control the zeroth mode. This point does not cause any difficulties because in this case $u_0$ is controlled by $u_1$ (due to solving \eqref{system}) which is controlled by $\|\cdot\|_Y$.
Let
\eqn{
Y=\{ u\in C([t_0,\infty);[0,\infty))^{\mathbb N+2-\sigma}:\| u\|_Y<\infty\}
}
with the $t_0$ defined above.

Consider the operator $T:Y\to Y$ defined by
\eqn{
T( u)_{k-1}&=N_{k-1}^\alpha\int_t^\infty\exp\Big(N_{k-1}^2(t'-t)-\int_{t}^{t'}N_{k-2}^\alpha u_{k-2}\Big)u_{k}^2(t')dt'\\
T( u)_k&=\delta_k e^{-N_k^2(t-\tau_k)}+\int_{t_0}^t\exp(-N_k^2(t-t'))(N_{k-1}^\alpha u_{k-1}u_k-N_k^\alpha u_{k+1}^2)(t')dt'.
}
The $u_{k-2}$ appearing in $T(u)_{k-1}$ must be specified in terms of $u$ in order to close the system. In the case $\sigma=1$, there is no issue because $u_{-1}=0$ by definition. When $\sigma=2$, we define $u_0(t)$ according to
\eq{\label{u0duhamel}
u_0(t)=e^{-N_0^2(t-t_0)}a-N_0^\alpha\int_{t_0}^t\exp(-N_0^2(t-t'))u_1^2(t')dt'.
}
It is clear that if $u$ is a fixed point of $T$ then it solves \eqref{system} on $[t_0,\infty)$ and has the correct value at $t=t_0$.

Define the closed ball in the $Y$ norm,
\eqn{
B\coloneqq\overline B_Y((0,\delta_ke^{-N_k^2(t-\tau_k)})_{k\in2\mathbb N+\sigma},1)
}
where the ball is centered at the point whose $j$th coordinate is either $0$ if $j\equiv\sigma-1\pmod2$ or $\delta_j\exp(-N_j^2(t-\tau_j))$ if $j\equiv\sigma\pmod2$.
We claim $T:B\to B$ and that it is a contraction. First note the useful fact that for all $u\in B$ and $j\in\mathbb N$, $u_j\geq0$. This is obvious from the definition of $T$ for $j\in2\mathbb N+\sigma-1$ and from the definition of $B$ for $j\in2\mathbb N+\sigma$. This only leaves the case $\sigma=2$, $j=0$. In this case it is easy to verify from \eqref{arange}, \eqref{u0duhamel}, and the fact that $u\in B$ that we have both the upper and lower bounds
\eqn{
0<\frac14\delta_1\exp(-N_0^2(t-t_0))\leq u_0(t)\leq 2\delta_1\exp(-N_0^2(t-t_0))
}
for all $t\geq t_0$. As a result $u_0$ possesses the same favorable properties as the other modes, as far as positivity and upper bounds are concerned.

Next we verify that $T$ maps $B$ to $B$. This is a straightforward calculation using the definition of $T$, the positivity of $u$, and the estimates entailed by the fact that $u\in B$. For the ``$k-1$'' indices, we have
\eqn{
|T(u)_{k-1}(t)|&\leq N_{k-1}^\alpha\int_t^\infty\exp(N_{k-1}^2(t'-t)-2N_k^2(t'-\tau_k))\delta_k^2\frac94dt'\leq3\delta_ke^{-2N_k^2(t-\tau_k)},
}
using the definition of $\delta_k$ to simplify the constants. For the ``$k$'' indices, we split the two terms in the nonlinearity in the definition of $T(u)_k$; using the triangle inequality, we bound $|T(u)_k-\delta_ke^{-N_k^2(t-\tau_k)}|$ by $I_1+I_2$. Then, by the same estimates for $u$, \eqref{tauconstraint}, and the fact that $t_0\sim N_0^{-2}\log\lambda\gg\tau_k$,
\eqn{
I_1&\leq\int_{t_0}^t\exp(-N_k^2(t-t')-3N_k^2(t'-\tau_k))dt'(N_{k-1}^\alpha\frac32\delta_k^2)\\
&\leq \frac34\delta_ke^{-N_k^2(t-t_0)-3N_k^2(t_0-\tau_k)}\\
&\leq\frac34\delta_ke^{-N_k^2(t-\tau_k)}\lambda^{-2\lambda^{2k}}\leq \frac14\delta_ke^{-N_k^2(t-\tau_k)}
}
and
\eqn{
I_{2}&\leq 9\int_{t_0}^t\exp(-N_k^2(t-t')-4N_{k+2}^2(t'-\tau_k))N_k^\alpha\delta_{k+2}^2\\
&\leq3\delta_ke^{-N_k^2(t-\tau_k)}\lambda^{4-3\alpha-3\lambda^{2(k+2)}}\leq \frac14\delta_ke^{-N_k^2(t-\tau_k)},
}
upon choosing $\lambda$ sufficiently large. The last three display inequalities taken together prove that $T(u)\in B$.

Now we show that $T$ is a contraction on $B$. Let $u,w\in B$. We implement a similar decomposition, letting
\eqn{
|T(u)_{k-1}-T(w)_{k-1}|\leq J_1+J_2
}
where
\eqn{
J_1&=N_{k-1}^\alpha\int_t^\infty\exp(N_{k-1}^2(t'-t))|u_k^2-w_k^2|(t')dt',\\
J_2&=N_{k-1}^\alpha\int_t^\infty\exp(N_{k-1}^2(t'-t))\int_t^{t'}N_{k-2}^\alpha|u_{k-2}-w_{k-2}|(t'')dt''w_k^2(t')dt'
}
are the two terms obtained when taking the variation in $u$, bounding $|e^{-A}-e^{-B}|\leq|A-B|$ ($A,B\geq0$). Note that $J_2$ identically vanishes when $k=1$.

It is straightforward to estimate
\eqn{
J_1&\leq \frac12N_{k-1}^\alpha\delta_k^2\|u-w\|_Y\int_t^\infty\exp(N_{k-1}^2(t'-t)-2N_k^2(t'-\tau_k))dt'\\&\leq \delta_k\|u-w\|_Ye^{-2N_k^2(t-\tau_k)}
}
and, as long as $k\notin\{1,2\}$,
\eqn{
J_2&\lesssim N_{k-1}^\alpha N_{k-2}^\alpha\delta_k^2\delta_{k-2}\|u-w\|_Y\\
&\quad\times\int_t^\infty\exp(N_{k-1}^2(t'-t)-2N_k^2(t'-\tau_k))\int_t^{t'}e^{-N_{k-2}^2(t''-\tau_{k-2})}dt''dt'\\
&\leq\delta_k\lambda^{\alpha}\exp(-2N_k^2(t-\tau_k)-N_{k-2}^2(t-\tau_{k-2}))\|u-w\|_Y.
}
Because $t\geq t_0\gg\tau_{k-2}$, we have $\exp(-N_{k-2}^2(t-\tau_{k-2}))\geq\exp(-N_{k-2}^2t_0/2)\geq\lambda^{O(-\lambda^2)}$. Thus by choosing $\lambda$ large, we can make $J_2$ much smaller than $\delta_k\exp(-2N_k^2(t-\tau_k))\|u-w\|_Y$.

When $\sigma=2$ and $k=2$, we cannot directly bound $|u_0-w_0|$ in terms of $\|u-w\|_Y$ when estimating $J_2$. Instead, it is very straightforward to verify from \eqref{u0duhamel} that the same bound applies, using $\|u-w\|_Y$ to estimate $|u_1-w_1|$.

Finally, we decompose
\eqn{
|T(u)_k-T(w)_k|\leq K_1+K_2
}
where
\eqn{
K_1&=N_{k-1}^\alpha\int_{t_0}^t\exp(-N_k^2(t-t'))|u_{k-1}u_k-w_{k-1}w_k|dt'\\
K_2&=N_k^\alpha\int_{t_0}^t\exp(-N_k^2(t-t'))|u_{k+1}^2-w_{k+1}^2|(t')dt'.
}
Then by calculations similar to the ones for $J_1,J_2$, we have
\eqn{
K_1\leq 3\delta_k\|u-w\|_Ye^{-N_k^2(t-\tau_k)}\lambda^{-2\lambda^{2k}}
}
and
\eqn{
K_2&\leq 6\delta_k\|u-w\|_Ye^{-N_k^2(t-\tau_k)}\lambda^{-3\lambda^{2(k+2)}+O(1)}.
}
Taken together, the bounds for $J_1,J_2,K_1,K_2$ imply that $\|T(u)-T(w)\|_Y/\|u-w\|_Y$ can be made, say, $\leq\frac12$ by taking $\lambda$ large. Thus we conclude existence of the solution by the Banach fixed point theorem. The desired bounds \eqref{manifoldboundsk-1} and \eqref{manifoldboundsk} follow directly from the fact that $u\in B$ along with the Duhamel formula for $u_{k-1}$ which proves positivity.
\end{proof}

\section{Bootstrapping the trajectories}\label{bootstrapsection}

Once again in this section, we fix a parity $\sigma\in\{1,2\}$ and let $k$ run over $2\mathbb N+\sigma$. We construct two separate solutions for $\sigma=1$ and $\sigma=2$, extending the ones from Section~\ref{stablemanifoldsection} backward in time.

We introduce the following ``good unknowns'' $r$ and $z$. By rewriting the system in these variables, we will find that the top order behavior becomes linear. The unknowns are defined as
\eq{\label{rzdefinition}
r_k=\sqrt{\Big(\frac{u_{k-1}}{\delta_{k}}-1\Big)^2+\Big(\frac{u_k}{\delta_k}\Big)^2},\quad z_k=\frac{\delta_k(r_k-1)+u_{k-1}}{u_k}.
}
$r$ and $z$ are best interpreted based on the polar coordinate system centered at $(\delta_k,0)$ in the space of $(u_{k-1},u_k)\in\mathbb R^2$. $r_k$ is the distance from $(\delta_k,0)$ rescaled by $\delta_k^{-1}$, while $z_k=\cot\frac{\theta_k}2$ where $\theta_k$ is the polar angle.  The approximate behavior can be described as $\dot r_k\approx0$ and $\dot z_k\approx-N_k^2z_k$ with error terms at least quadratic in $r_k-1$ and $z_k$. Roughly speaking, the idea of the bootstrap is that when $N_k^2t$ is large, $z_k$ and $r_k-1$ are exponentially small; when $N_k^2t$ is small, there is not enough time for the errors to accumulate.

For brevity, we define the rational functions
\begin{equation}\label{rationalfunctions}
\begin{gathered}
    R_1(z)=\frac{z^2-1}{z^2+1},\quad R_2(z)=-\frac{2z^2(z^2-1)}{(z^2+1)^2},\quad R_3(z)=\frac{2z}{z^2+1},\\
    R_4(z)=-\frac{2z^2}{z^2+1},\quad R_5(z)=z^2-1.
\end{gathered}
\end{equation}
One verifies the elementary bounds
    \begin{equation}\label{Rrationalfunctionbounds}\begin{gathered}
    |R_1(z)|\leq 1,\quad |R_2(z)|,|R_4(z)|\leq2\min(z^2,1),\quad |R_3(z)|\leq\min(1,2z),\\
    |R_1'(z)|,|R_2'(z)|,|R_4'(z)|\leq\min(4z,2),\quad 
    |R_3'(z)|\leq2,\quad |R_5'(z)|\leq2z
    \end{gathered}\end{equation}
for all $z\geq0$. Now we record the following useful formulas which follow from trigonometry:
\eq{\label{ufromrandz}
u_{k-1}=\delta_k(1+R_1(z_k)r_k),\quad u_k=\delta_kr_kR_3(z_k).
}
It is easy to see that there is a one-to-one correspondence
 \eqn{
 \{(r,z)\in(0,\infty)^{(2\mathbb N+\sigma)\times2}\}\leftrightarrow\{u\in\mathbb R^{\mathbb N+\sigma-1}:u_k>0\text{ for all }k\in2\mathbb N+\sigma\}.
 }
 In other words, when $\sigma=1$, the $(r,z)$ coordinates completely describe the space of possible positive solutions of \eqref{system}. On the other hand, when $\sigma=2$, all but $u_0$ are determined. Here and elsewhere, when $r$ and $z$ appear without an index, they denote the full collections $(z_k)_{k\in2\mathbb N+\sigma}$ and $(r_k)_{k\in2\mathbb N+\sigma}$.

One computes\footnote{To perform this computation, it is quite a bit easier to first compute the system for the (rescaled) polar coordinates $r_k,\theta_k$. One notices that the quantity $\cot(\theta_k/2)$  appears often, and it is then straightforward to compute $\dot z_k$ from $\dot\theta_k$.} that these variables obey
\begin{equation}\begin{aligned}\label{rzsystem}
\dot r_k&=b_k(t)(-R_2(z_k)+R_1^2(z_k)(r_k-1))+\delta_k^{-1}R_3(z_k)f_k(t)\\
\dot z_k&=-N_k^2r_kz_k+b_k(t)z_k\left(-R_4(z_k)+r_k^{-1}-1\right)+(2\delta_kr_k)^{-1}R_5(z_k)f_k(t)
\end{aligned}\end{equation}
where $b_k=N_{k-2}^\alpha u_{k-2}-N_{k-1}^2$ and $f_k=N_k^\alpha u_{k+1}^2$.

For $\sigma\in\{1,2\}$ and $I\subset[0,\infty)$, we define the metric
\eqn{
d_{\sigma,I}((r,z),(s,w))&=\sup_{t\in I}\sup_{k\in2\mathbb N+\sigma}\big(w_k^r(t)|r_k(t)-s_k(t)|+w_k^z(t)|z_k(t)-w_k(t)|\big)
}
on the space
\eqn{
B_{\sigma,I}(R)=\{(r,z)\in(0,\infty)^{(2\mathbb N+\sigma)\times2}:d_{\sigma,I}((r,z),(1,\frac12e^{-N_k^2(t-\tau_k)})_{k\in2\mathbb N+\sigma})\leq R\}
}
where the weights are given by
\eqn{
w_k^r(t)=\max(\lambda^{\frac12(4-\alpha)},\exp(2N_k^2(t-\tau_k))),\quad w_k^z(t)=30\exp(N_k^2(t-\tau_k)).
}

Propositions~\ref{bootsrapboundsprop}--\ref{bootstrapconclusionprop} will bootstrap bounds backward from $t=t_0$ to $t=0$ in terms of the $r,z$ unknowns. But first, we must show that the bounds proved at $t_0$ in Proposition~\ref{stablemanifoldprop} suffice to initialize the bootstrap. Note that the arguments below are applied to the solution coming from Proposition~\ref{stablemanifoldprop} corresponding to some fixed choice of $\sigma\in\{1,2\}$, $\tau$ obeying \eqref{tauconstraint}, and $a$ obeying \eqref{arange}.

\begin{lem}\label{utorzlemma}
    Let $u$ be a solution of \eqref{system} on $[t_0,\infty)$ constructed in Proposition~\ref{stablemanifoldprop}. Expressing $u$ in coordinates $(r,z)$, it obeys $(r,z)\in B_{\sigma,[t_0,\infty)}({10})$.

    Moreover at $t_0$ we have the improved bound
\eqn{
    |z_k(t_0)-\frac12\exp(-N_k^2(t_0-\tau_k))|&=O(\exp(-3N_k^2(t_0-\tau_k))).
    }
\end{lem}

\begin{proof}
    We use the Taylor expansion to compute
    \eqn{
    r_k-1&=\sqrt{\Big(\frac{u_{k-1}}{\delta_k}-1\Big)^2+\Big(\frac{u_k}{\delta_k}\Big)^2}-1=-\frac{u_{k-1}}{\delta_k}+\frac12\Big(\frac{u_k}{\delta_k}\Big)^2+O\Big(\Big(\frac{u_{k-1}}{\delta_k}\Big)^2+\Big(\frac{u_{k}}{\delta_k}\Big)^4\Big).
    }
    In fact, the big-O is not worst than a factor of $2$. Therefore for all $t\geq t_0$
    \eqn{
    |r_k(t)-1|\leq10\exp(-2N_k^2(t-\tau_k)).
    }
    Next, plugging the Taylor expansion into the formula for $z_k$, we compute that for all $t\geq t_0$,
    \eqn{
    z_k-\frac12\exp(-N_k^2(t-\tau_k))&=\frac{\delta_k(r_k-1)+u_{k-1}}{u_k}-\frac12\exp(-N_k^2(t-\tau_k))\\
    &=\frac12\delta_k^{-1}(u_k-\delta_k\exp(-N_k^2(t-\tau_k)))\\
    &\quad+\delta_kO(u_k^{-1}\exp(-4N_k^2(t-\tau_k))).
    }
    In the special case $t=t_0$, the first term identically vanishes. The estimate in the general case $t\geq t_0$ follows in a straightforward way from \eqref{manifoldboundsk-1} and \eqref{manifoldboundsk}.
\end{proof}

Now we introduce the bootstrap hypotheses.

\begin{define}[Weak bootstrap hypothesis $\mathcal H$]
For a $t\in[0,t_0]$, we say that $u$ obeys $\mathcal H(t)$ if
\begin{itemize}
    \item $u$ exists and solves \eqref{system} on $[t,\infty)$,
    \item $u$ agrees with the solution constructed in Proposition~\ref{stablemanifoldprop} on $[t_0,\infty)$, and
    \item we have $(r,z)\in B_{\sigma,[t,t_0]}({10})$.
\end{itemize}
   
\end{define}

\begin{define}[Strong bootstrap hypothesis $\mathcal H'$]
    We say that $u$ obeys $\mathcal H'(t)$ if in addition to $\mathcal H(t)$, we have $(r,z)\in B_{\sigma,[t,t_0]}(\eps)$.
\end{define}

We carry out a bootstrap argument, initialized with the known statement $\mathcal H(t_0)$, which we propagate to the desired statement $\mathcal H'(0)$. The first step is to show that thanks to nonlinear effects, the weak statement actually implies the strong statement on the same time interval.

\begin{prop}[Self-improving bounds]\label{bootsrapboundsprop}
    Let $u$ satisfy the hypothesis $\mathcal H(t_1)$ for some $t_1\in[0,t_0]$. Then it also satisfies $\mathcal H'(t_1)$.
\end{prop}

\begin{proof}
    By assuming $\mathcal H(t_1)$ we have access to the bounds
    \eq{\label{rzinhypothesisball}
    |r_k(t)-1|&\leq10\min(\lambda^{-\frac12(4-\alpha)},e^{-2N_k^2(t-\tau_k)}),\quad |z_k(t)-\frac12e^{-N_k^2(t-\tau_k)}|\leq\frac13e^{-N_k^2(t-\tau_k)}
    }
    for all $k\in2\mathbb N+\sigma$ and $t\in[t_1,t_0]$. The objective is to prove that the same bounds hold with a smaller constant.
    
    In a moment we will justify the following Duhamel formula: fix $\sigma\in\{1,2\}$ and let $r,z$ be the good unknowns corresponding to $u$. Then for all $t\in[0,t_0]$ and $k\in2\mathbb N+\sigma$,
    \eqn{
    &r_k(t)-1=\int_t^\infty\exp(\int_t^{t'}b_kR_1^2(z_k)dt'')\left(b_kR_2(z_k)+\frac{N_k^\alpha}{\delta_k}R_3(z_k)u_{k+1}^2\right)dt'\\
    &z_k(t)=e^{N_k^2(t_0-t)}z_k(t_0)\\
    &\quad+\int_{t}^{t_0}e^{N_k^2(t'-t)}\Big(b_kR_4(z_k)+(N_k^2-b_kr_k^{-1})(r_k-1)z_k+\frac{N_k^\alpha}{2\delta_k}r_k^{-1}R_5(z_k) u_{k+1}^2\Big)dt'.
    }
    The $u_{k-2}$ and $u_{k+1}$ entering the formula can be expressed in terms of $r$ and $z$ by way of \eqref{ufromrandz} with the exception of $u_{0}$ when $\sigma=k=2$. In this case we instead make use of the Duhamel formula
    \eq{
    u_0(t)=e^{N_0^2(t_0-t)}a-N_0^\alpha\delta_2^2\int_t^{t_0}e^{N_0^2(t'-t)}\left(R_4(z_2)+(r_2-1)R_1(z_2)\right)^2(t')dt'\label{u0duhamelII}
    }
    coming from $\dot u_0=-N_0^2u_0-N_0^\alpha u_1^2$ as in \eqref{system} and \eqref{ufromrandz}. (Recall the meaning of $a$ from Proposition~\ref{stablemanifoldprop}.) We will make use of the property that\footnote{\label{alpha4footnote}We remark that this bound for $b$ is essentially optimal globally but not on the relevant time scale, leading to the unnecessary restriction $\alpha<4$. Indeed, for $t\ll N_{k-2}^{-2}$, $u_{k-2}\approx u_{k-2}(0)\approx2\delta_{k-1}$ which implies $|b_k|\lesssim N_{k-1}^2$. This improved bound would suffice to prove the theorem for all $\alpha>2$.}
    \eq{
    \sup_{t\in[t_1,t_0]}|b_k|\leq2N_{k-2}^\alpha\delta_{k-2}\leq2\lambda^{-4+\alpha}N_k^2\label{bbound}
    }
    for all $k\in2\mathbb N+\sigma$. For $k\neq2$, \eqref{bbound} is immediate from \eqref{ufromrandz} and \eqref{rzinhypothesisball}. This upper bound dominates the $N_{k-1}^2$ term as long as $\alpha>2$ and $\lambda$ is chosen sufficiently large.

    When $k=2$, the bound on $u_{0}$ and therefore $b_2$ must be extracted from \eqref{u0duhamelII}. A straightforward calculation yields $|u_0(t)|\leq 3N_1^{-\alpha+2}$ which suffices to conclude \eqref{bbound} when $k=2$.

    To justify the Duhamel formula above, it is necessary to verify that the ``boundary term'' from $t=\infty$ in the formula for $r_k-1$ vanishes. Namely,
    \eqn{
    \lim_{T\to\infty}\exp(-\int_t^Tb_kR_1(z_k)^2)(r_k(T)-1)\lesssim\lim_{T\to\infty}\exp((T-t)(\sup_t|b_k|)-2N_k^2(T-\tau_k))
    }
    which vanishes, using \eqref{Rrationalfunctionbounds}, \eqref{rzinhypothesisball}, and \eqref{bbound}.
    
    Finally we can proceed with estimating the $d_{\sigma,[t_1,t_0]}$ distance between $(r_k,z_k)$ and $(1,\frac12e^{-N_k^2(t-\tau_k)})$ to prove the desired inclusion. Using the triangle inequality we decompose $|r_k-1|\leq I_1+I_2$, the right-hand side being the two nonlinear terms in the Duhamel formula. Since the metric contains a weight $w_k^r$ that is the maximum of two possible values, we must consider two different cases depending on whether we are in the dissipative range. Assume first that $t\geq\tau_k$. Then by \eqref{Rrationalfunctionbounds}, \eqref{rzinhypothesisball}, and \eqref{bbound}, we have
    \eqn{
    I_1&\lesssim\|b_k\|_{L^\infty}\int_t^\infty\exp((t'-t)\|b_k\|_{L^\infty}) z_k(t')^2dt'\\
    &\lesssim\|b_k\|_{L^\infty}\int_t^\infty\exp((t'-t)\|b_k\|_{L^\infty}-2N_k^2(t'-\tau_k))dt'\\
    &\lesssim\lambda^{-4+\alpha}e^{-2N_k^2(t-\tau_k)}.
    }

    Alternately, for any $t\in[t_1,t_0]$, we can estimate using the previous bound for the tail part of the integral, \eqref{Rrationalfunctionbounds}, \eqref{rzinhypothesisball}, and \eqref{bbound}, that
    \eqn{
    I_1&\lesssim\lambda^{-4+\alpha}+\|b_k\|_{L^\infty}\int_t^{\tau_k}\exp((t'-t)\|b_k\|_{L^\infty})dt'1_{t\leq\tau_k}\\
    &\lesssim\lambda^{-4+\alpha}+e^{\tau_k\|b_k\|_{L^\infty}}-1\\
    &\lesssim\lambda^{-4+\alpha}\log\lambda.
    }

    Now consider $I_2$. To estimate $u_{k+1}$, we must express it in terms of $r$ and $z$. From \eqref{rationalfunctions} and \eqref{ufromrandz}, one computes
    \eqn{
    u_{k-1}=\delta_k(1+r_k\frac{z_k^2-1}{z_k^2+1})\leq\delta_k(|r_k-1|+2\min(z_k^2,1))
    }
    and therefore
    \eq{
    |u_{k+1}|&\leq\delta_{k+2}\min(\exp(-2N_{k+2}^2(t-\tau_{k+2})),\frac1{10})+2\delta_{k+2}\min(\exp(-2N_{k+2}^2(t-\tau_{k+2})),1)\nonumber\\
    &\lesssim\delta_{k+2}\min(\exp(-2N_{k+2}^2(t-\tau_{k+2})),1)\label{uk+1bound}
    }
    for $t\in[t_1,t_0]$. In fact it holds on $[t_1,\infty)$ by \eqref{manifoldboundsk-1}.
    
    We split into cases as above. If $t\geq\tau_k$,
    \eqn{
    |I_2|&\lesssim N_k^\alpha\delta_k^{-1}\int_t^\infty\exp(\|b_k\|_{L^\infty}(t'-t))|z|u_{k+1}^2(t')dt'\\
    &\lesssim \lambda^{4-3\alpha}N_{k+2}^2\int_t^\infty\exp(\|b_k\|_{L^\infty}(t'-t)-N_k^2(t'-\tau_k)-4N_{k+2}^2(t'-\tau_{k+2}))dt'\\
    &\lesssim\lambda^{4-3\alpha}\exp(-N_k^2(t-\tau_k)-4N_{k+2}^2(t-\tau_{k+2}))\\
    &\lesssim\lambda^{-(\alpha-2)\lambda^{4}+4-3\alpha}\exp(-2N_k^2(t-\tau_k))
    }
    by \eqref{Rrationalfunctionbounds}, \eqref{rzinhypothesisball}, and \eqref{bbound}. Note that $t-\tau_{k+2}\geq\tau_k-\tau_{k+2}\geq\frac12N_k^{-2}\log\lambda^{\alpha-2}$ by \eqref{tauconstraint}. For all $t\leq t_0$, we could instead estimate, using the other upper bound for $R_3$ in \eqref{Rrationalfunctionbounds},
    \eqn{
    |I_2|&\lesssim N_k^\alpha\delta_k^{-1}\int_t^\infty\exp(\|b_k\|_{L^\infty} (t'-t))u_{k+1}^2(t')dt'
    }
    and separately consider integration over $[\tau_k,\infty)$, $[2\tau_{k+2},\tau_k]$, and $[0,2\tau_{k+2}]$, to the extent these domains overlap with $[t,\infty)$. We name these pieces $I_{2,1},\,I_{2,2},\,I_{2,3}$ respectively. By the previous computation, $I_{2,1}\leq \lambda^{-(\alpha-2)\lambda^{4}}$.
    
    Using the same ingredients, we also have
    \eqn{
    I_{2,2}&\leq N_k^\alpha\delta_k^{-1}\delta_{k+2}^2\int_{2\tau_{k+2}}^{\tau_k}\exp(\|b_k\|_{L^\infty} t'-4N_{k+2}^2(t'-\tau_{k+2}))\\
    &\lesssim\lambda^{-3\alpha+4}\exp(-(4N_{k+2}^2-2\|b_k\|_{L^\infty})\tau_{k+2})\\
    &\lesssim\lambda^{-6\alpha+10}
    }
    and
    \eqn{
    I_{2,3}&\lesssim N_k^\alpha\delta_k^{-1}\int_0^{2\tau_{k+2}}\exp(\|b_k\|_{L^\infty} t')u_{k+1}^2(t')dt'\\
    &\lesssim N_k^\alpha\delta_k^{-1}\delta_{k+2}^2\|b_k\|_{L^\infty}^{-1}(e^{2\|b_k\|_{L^\infty}\tau_{k+2}}-1)\\
    &\lesssim \lambda^{-6\alpha+10}\log\lambda.
    }
    In summary, we have shown that for all $t\in[t_1,t_0]$, $r$ obeys $|r_k(t)-1|\lesssim \lambda^{-4+\alpha}\log\lambda$ which implies
    $\lambda^{\frac12(4-\alpha)}|r_k(t)-1|\leq1$ upon taking $\lambda$ sufficiently large. If $\exp(2N_k^2(t-\tau_k))$ is the dominant part of the weight $w_k^r(t)$, clearly it must be that $t\geq\tau_k$. In this case, collecting our bounds for $I_1$ and $I_2$, we have already proved $w_k^r|r_k(t)-1|\leq1$ upon taking $\lambda$ sufficiently large.
    
    Next we decompose
    \eqn{
    z_k-\frac12\exp(-N_k^2(t-\tau_k))=J_1+J_2+J_3+J_4
    }
    where
    \eqn{
    J_1=\exp(N_k^2(t_0-t))z_k(t_0)-\frac12\exp(-N_k^2(t-\tau_k))
    }
    and $J_2,J_3,J_4$ are the nonlinear terms in the Duhamel formula for $z_k$. Then by Lemma~\ref{utorzlemma} and \eqref{tauconstraint}, the linear term satisfies
    \eqn{
    |J_1|&\leq\exp(N_k^2(t_0-t))|z_k(t_0)-\frac12\exp(-N_k^2(t_0-\tau_k))|\\
    &\lesssim\exp(N_k^2(t_0-t)-3N_k^2(t_0-\tau_k))\\
    &=\exp(-2N_k^2(t_0-\tau_k))\exp(-N_k^2(t-\tau_k))\\
    &\leq\lambda^{-\lambda^{2}}w_k^z(t)^{-1}.
    }
    For the nonlinear terms, we need to split into cases depending on the size of $tN_k^2$. If $t\geq2\tau_k$ dissipation is the dominant effect and we have by \eqref{rationalfunctions}, \eqref{bbound}, \eqref{rzinhypothesisball}, and \eqref{uk+1bound} that
    \eqn{
    |J_2|&\lesssim\|b_k\|_{L^\infty}\int_t^{t_0}\exp(N_k^2(t'-t))z_k^2(t')dt'\\
    &\lesssim \lambda^{-4+\alpha}N_k^2\int_t^{t_0}\exp(N_k^2(t'-t)-2N_k^2(t'-\tau_k))dt'\\
    &\lesssim \lambda^{-4+\alpha}\exp(-2N_k^2(t-\tau_k)),
    }
    \eqn{
    |J_3|&\lesssim N_k^2\int_t^{t_0}\exp(N_k^2(t'-t)-3N_k^2(t'-\tau_k))dt'\\
    &\lesssim\exp(-3N_k^2(t-\tau_k))\\
    &\lesssim\lambda^{-2(\alpha-2)}\exp(-N_k^2(t-\tau_k)),
    }
    and
    \eqn{
    |J_4|&\leq N_k^\alpha\delta_k^{-1}\int_t^{t_0}e^{N_k^2(t'-t)}|z_k^2(t')-1|u_{k+1}^2(t')dt'\\
    &\lesssim\delta_{k+2}^2N_k^\alpha\delta_k^{-1}\int_t^{t_0}\exp(N_k^2(t'-t)-4N_{k+2}^2(t'-\tau_{k+2}))dt'\\
    &\lesssim N_{k+2}^{-2}\delta_{k+2}^2N_k^\alpha\delta_k^{-1}\exp(-4N_{k+2}^2(t-\tau_{k+2}))\\
    &\leq\lambda^{-3\alpha+4}\exp(-4N_{k+2}^2(\tau_k-\tau_{k+2}))\exp(-N_k^2(t-\tau_k))\\
    &\leq\lambda^{-3\alpha+4-2(\alpha-2)\lambda^{4}}\exp(-N_k^2(t-\tau_k)).
    }

    Suppose instead $t\leq2\tau_k$. Then we estimate
    \eqn{
    |J_2|&\lesssim \lambda^{-4+\alpha}+\|b_k\|_{L^\infty}\int_t^{\tau_k}\exp(N_k^2(t'-t))dt'\lesssim\lambda^{-4+\alpha}(1+\exp(N_k^2(\tau_k-t))),
    }
    \eqn{
    |J_3|&\lesssim \lambda^{-2(\alpha-2)}+N_k^2\lambda^{-\frac12(4-\alpha)}\int_t^{2\tau_k}\exp(N_k^2(t'-t)-N_k^2(t'-\tau_k))\\
    &\lesssim\lambda^{-2(\alpha-2)}+N_k^2\tau_k\lambda^{-\frac12(4-\alpha)}\exp(N_k^2(\tau_k-t)),
    }
    and
    \eqn{
    |J_4|&\lesssim \lambda^{-3\alpha+4-2(\alpha-2)\lambda^{4}}+N_k^\alpha\delta_k^{-1}\lambda^{2\alpha-4}\int_0^{2\tau_k}\exp(N_k^2t')u_{k+1}^2(t')dt'
    }
using the calculation for $t\geq2\tau_k$ to bound the tail part of the integrals, \eqref{rzinhypothesisball}, and the fact that
\eqn{
z_k^2\leq\exp(2N_k^2\tau_k)\lesssim\lambda^{2(\alpha-2)}
}
by \eqref{tauconstraint}. The latter term in the $J_4$ upper bound can be split into two pieces:
    \eqn{
    &N_k^\alpha\delta_k^{-1}\lambda^{2\alpha-4}\int_{2\tau_{k+2}}^{2\tau_k}\exp(N_k^2t')u_{k+1}^2(t')dt'\\
    &\quad\lesssim N_k^\alpha\delta_k^{-1}\lambda^{2\alpha-4}\delta_{k+2}^2\int_{2\tau_{k+2}}^{2\tau_k}\exp(N_k^2t'-4N_{k+2}^2(t'-\tau_{k+2}))dt\\
    &\quad\lesssim N_k^\alpha\delta_k^{-1}\delta_{k+2}^2N_{k+2}^{-2}\lambda^{2\alpha-4}\exp(2N_k^2\tau_{k+2})\\
    &\quad\lesssim\lambda^{-\alpha}
    }
    and
    \eqn{
    N_k^\alpha\delta_k^{-1}\lambda^{2\alpha-4}\int_0^{2\tau_{k+2}}\exp(N_k^2t')u_{k+1}^2(t')dt'&\lesssim \delta_{k+2}^2N_k^\alpha\delta_k^{-1}\tau_{k+2}\lambda^{2\alpha-4}\exp(2N_k^2\tau_{k+2})\\
    &\lesssim\lambda^{-\alpha}\log\lambda.
    }
    Collecting all the estimates on $J_1,J_2,J_3,J_4$, we have shown that
    \eqn{
    w_k^z(t)|z_k-\frac12e^{-N_k^2(t-\tau_k)}|&\lesssim \lambda^{-c}+\lambda^{-2(\alpha-2)}e^{N_k^2(t-\tau_k)}1_{t\leq2\tau_k}
    }
    for some $c>0$, the dominant contribution to the second term coming from $J_3$. Note that as long as $t\leq2\tau_k$, the exponential factor in this term is $O(\lambda^{\alpha-2})$ so the whole right-hand side can be made arbitrarily small with the choice of $\lambda$. This completes the proof that $r,z\in B_{\sigma,[t_1,t_0]}(\eps)$
\end{proof}

Now we state the other essential ingredient in the bootstrap argument.

\begin{prop}[Propagation backwards]\label{localprop}
    If $u$ obeys the hypothesis $\mathcal H'(t_1)$ for some $t_1\in(0,t_0]$, then there exists $t_2\in(0,t_1)$ such that we have $\mathcal H(t_2)$.

    Moreover, the backward continuation of the solution with these bounds is unique.
\end{prop}

\begin{proof}
    The proof is by a contraction mapping argument. Consider the operator $(r,z)\mapsto(F(r,z),G(r,z))$ defined by
    \eqn{
    &F_k(r,z)(t)=1+\exp(\int_{t}^{t_1}b_kR_1^2(z_k)dt')(r_k(t_1)-1)\\
    &\quad+\int_t^{t_1}\exp(\int_t^{t'}b_kR_1^2(z_k)dt'')\left(b_kR_2(z_k)+\frac{N_k^\alpha}{\delta_k}R_3(z_k)u_{k+1}^2\right)dt'\\
    &G_k(r,z)(t)=e^{N_k^2(t_1-t)}z_k(t_1)\\
    &\quad+\int_{t}^{t_1}e^{N_k^2(t'-t)}\Big(b_kR_4(z_k)+(N_k^2-b_kr_k^{-1})(r_k-1)z_k+\frac{N_k^\alpha}{2\delta_k}r_k^{-1}R_5(z_k) u_{k+1}^2\Big)dt'
    }
    for $t\in[t_2,t_1]$, where $t_2=t_1-T$ and $T\ll_\lambda t_1$ is to be determined. There is mild abuse of notation: the $r_k(t_1)$ and $z_k(t_1)$ appearing in the linear terms are coming from the known initial data rather than the arguments of $F,G$. Indeed, $u$ is already determined on $[t_1,\infty)$ by the assumption $\mathcal H'(t_1)$. The distinction is not significant because $r(t_1)$ is fixed by $F$.
    
    We claim that the $(F,G)$ operator is well-defined. It is straightforward to define $u_{k+1}$ in terms of $r$ and $z$ by means of \eqref{ufromrandz}. When $k\neq2$, the drift $b_k$ is defined by the usual formula containing $u_{k-2}$ which is obtained from $(r,z)$ according to \eqref{ufromrandz}. When $\sigma=k=2$, $b$ contains $u_0$ which is not determined by $(r,z)$ but is instead computed from
    \eq{
    u_0(t)=e^{N_0^2(t_1-t)}u_0(t_1)-N_0^\alpha\delta_2^2\int_t^{t_1}e^{N_0^2(t'-t)}\left(R_4(z_2)+(r_2-1)R_1(z_2)\right)^2(t')dt'\label{u0duhamelIII}
    }
    analogously to \eqref{u0duhamel} and \eqref{u0duhamelII}. Note that once again the quantity $u_0(t_1)$ in the linear term comes upon assuming the hypothesis $\mathcal H'(t_1)$. From here we can recover the bound on $b$ \eqref{bbound}.
    
    We claim that $(F,G):B_{\sigma,[t_2,t_1]}({10})\to B_{\sigma,[t_2,t_1]}({10})$ and that it is a contraction in the $d_{\sigma,[t_2,t_1]}$ metric when $T= t_1-t_2$ is chosen sufficiently small. Assume that $(r,z)$ defined on $[t_2,t_1]$ has
    \eq{\label{rzintheball}
    (r,z)\in B_{\sigma,[t_2,t_1]}(10).
    }
    With the triangle inequality we decompose $|F_k(r,z)-1|\leq I_1+I_2+I_3$ which are the respective absolute values of the linear contribution to $F_k-1$ and the two nonlinear terms. Because of the weight $w_k^r(t)$, we claim an upper bound of $10\min(\lambda^{-\frac12(4-\alpha)},e^{-2N_k^2(t-\tau_k)})$ for $I_1+I_2+I_3$. Thus it is necessary to separately consider the ranges of $k$ for which each possible minimum is realized. Let us define $N_*^2=t_1^{-1}\log\lambda$. It is easy to verify that $N_k\geq N_*$ if the minimum is achieved by $e^{-2N_k^2(t-\tau_k)}$, while $N_k\leq2N_*$ if the minimum is achieved by $\lambda^{-\frac12(4-\alpha)}$.

    Let us fix $t_1$ and assume at the outset that $T\leq t_1/2$. We may make $T$ as small as needed, as long as it depends only on $t_1$ and $\lambda$ but not $k$. To bound $I_1$, it is essential that we have the stronger bootstrap hypothesis $\mathcal H'(t_1)$ at time $t_1$. From \eqref{Rrationalfunctionbounds}, \eqref{bbound}, and the bound on $r(t_1)$ from $(r,z)\in B_{\sigma,[t_1,t_0]}(\eps)$, we have that
    \eqn{
    |I_1|&\leq\exp(\|b_k\|_{L^\infty}(t_1-t))|r(t_1)-1|\\
    &\leq\eps\exp(\lambda^{-4+\alpha}N_k^2(t_1-t))\min(\lambda^{-\frac{4-\alpha}2},\exp(-2N_k^2(t_1-\tau_k)))\\
    &\leq\eps\min(\lambda^{-\frac{4-\alpha}2}\exp(T\lambda^{-4+\alpha}N_k^2),\exp(-2N_k^2(t-\tau_k)).
    }
    If $N_k\leq2N_*$ then we bound $|I_1|$ by the first part of the minimum. In this case $T\lambda^{-4+\alpha}N_k^2\leq4(T/t_1)\lambda^{-4+\alpha}\log\lambda$. As long as $T\leq t_1$ and $\lambda$ is large, the exponential is bounded by $2$. If instead $N_k\geq N_*$, then we bound $|I_1|$ by the second part of the minimum. Thus in either case we arrive at
    \eqn{
    |I_1|\leq\min(\lambda^{-4+\alpha},\exp(-2N_k^2(t-\tau_k)))\leq\frac1{w_k^r(t)}
    }
    for all $k\in2\mathbb N+\sigma$. For $I_2$, we observe by a simple calculation from \eqref{Rrationalfunctionbounds}, \eqref{bbound}, and the bounds entailed by \eqref{rzintheball} that for all $k\in2\mathbb N+\sigma$,
\eqn{
|I_2|\leq(1/10)\exp(-2N_k^2(t-\tau_k))
}
upon taking $\lambda$ sufficiently large. When $N_k\leq2N_*$, we additionally have
\eqn{
|I_2|&\leq2(\exp(\|b_k\|_{L^\infty}(t_1-t))-1)\leq 10\lambda^{4-\alpha}(\log\lambda)\frac T{t_1}\leq\frac1{10}\lambda^{-\frac12(4-\alpha)}
}
by \eqref{Rrationalfunctionbounds}, \eqref{rzintheball}, and Taylor's theorem, upon choosing $T/t_1$ sufficiently small (depending on $\lambda)$. We perform the same procedure for $I_3$. Assume that $k\in2\mathbb N+\sigma$ is so large that $N_k\geq N_*$. Then, by the same list of ingredients,
\eqn{
|I_3|&\lesssim N_k^\alpha\delta_k^{-1}\int_t^{t_1}\exp(\|b_k\|_{L^\infty}(t'-t))u_{k+1}^2dt'\\
&\lesssim N_k^\alpha\delta_k^{-1}\delta_{k+2}^2\int_t^{t_1}\exp(\|b_k\|_{L^\infty}(t'-t)-4N_{k+2}^2(t'-\tau_{k+2}))dt'\\
&\lesssim \lambda^{12-4\alpha}\exp(-2N_k^2(t-\tau_k))\exp(-4N_{k+2}^2(t-\tau_{k+2})+2N_k^2(t-\tau_k))\\
&\lesssim \lambda^{12-4\alpha}\exp(-2N_k^2(t-\tau_k))\exp(-4N_{k+2}^2(\frac12N_k^{-2}\log\lambda-\tau_{k+2})+2N_k^2(N_k^{-2}\log\lambda-\tau_k))\\
&\leq\lambda^{-\lambda^{4}+O(1)}\exp(-2N_k^2(t-\tau_k))
}
where we have used crucially that $t\geq\frac12t_1\geq\frac12N_k^{-2}\log\lambda$ on $[t_2,t_1]$. On the other hand, if $N_k\leq2N_*$, then
\eqn{
|I_3|&\lesssim N_k^\alpha\delta_k^{-1}\delta_{k+2}^2\int_t^{t_1}\exp(\|b_k\|_{L^\infty}(t'-t))dt'\\
&\lesssim N_k^\alpha\delta_k^{-1}\delta_{k+2}^2\|b_k\|_{L^\infty}^{-1}(\exp(T\|b_k\|_{L^\infty})-1)\\
&\lesssim \lambda^{-3\alpha+8}N_k^2T\\
&\lesssim \lambda^{-3\alpha+8}(\log\lambda) \frac{T}{t_1}
}
which can be made arbitrarily small with the choice of $T$. Collecting the bounds for $I_1$--$I_3$ in the two regimes for $k$, we conclude that $w_k^r|F_k-1|$ is uniformly bounded by $10$.

Similarly, we use the triangle inequality to decompose $|G_k(t)-\frac12\exp(-N_k^2(t-\tau_k)|\leq J_1+J_2+J_3+J_4$ where
\eqn{
J_1=\big|\exp(N_k^2(t_1-t))z(t_1)-\frac12\exp(-N_k^2(t-\tau_k))\big|
}
and $J_2$--$J_4$ are the three nonlinear terms in the formula for $G_k$. It is immediate from the bound on $r(t_1)$ coming from $\mathcal H'(t_1)$ that
\eqn{
J_1\leq\frac1{3000}\exp(-N_k^2(t-\tau_k)),
}
so the contribution of $w_k^zJ_1$ to the metric is bounded by $\eps$. We omit the details of the straightforward calculation for the remaining terms which can be estimated using the same elements as for $F$: \eqref{Rrationalfunctionbounds}, \eqref{bbound}, \eqref{rzintheball}, $\eqref{tauconstraint}$, and decomposition into small and large $N_k$. With the appropriate large choice of $\lambda$ and small choice of $T/t_1$, one arrives at $w_k^z(J_2+J_3+J_4)\leq1$ for all $t\in[t_2,t_1]$ and $k\in2\mathbb N+\sigma$.

To complete the proof we argue that $(F,G)$ is a contraction on $B_{\sigma,[t_2,t_1]}(10)$. In order to express the estimates compactly we write
\eqn{
F_k(r,z)&=f_k(t,z_k,b_k)+\int_t^{t_1}g_k(t',t,z_k,b_k,u_{k+1})dt',\\
G_k(r,z)&=c_k(t)+\int_t^{t_1}h_k(t',t,r_k,z_k,b_k,u_{k+1})dt'
}
where
\eqn{
f_k(t,z_k,b_k)&=1+\exp(\int_{t}^{t_1}b_kR_1^2(z_k)dt')(r_k(t_1)-1),\\
g_k(t',t,z_k,b_k,u_{k+1})&=\exp(\int_t^{t'}b_kR_1^2(z_k)dt'')\left(b_kR_2(z_k)+\frac{N_k^\alpha}{\delta_k}R_3(z_k)u_{k+1}^2\right),\\
h_k(t',t,r_k,z_k,b_k,u_{k+1})&=e^{N_k^2(t'-t)}\Big(b_kR_4(z_k)+(N_k^2-b_kr_k^{-1})(r_k-1)z_k+\frac{N_k^\alpha}{2\delta_k}r_k^{-1}R_5(z_k) u_{k+1}^2\Big).
}
Since $c_k$ does not depend on $r,z$ it does not concern us here. We will be somewhat informal and take the variation of $F_k,G_k$ with respect to $r_k$, $z_k$, etc.\ by differentiating $f_k,g_k,h_k$. It would be more precise to express this in terms of Fr\'echet differentiation, but the result would be identical.

Let $(r,z),\,(r',z')\in B_{\sigma,[t_2,t_1]}(10)$. By convexity, we can bound the difference between $(F,G)(r,z)$ and $(F,G)(r',z')$ by the maximum of the variation with respect to the many inputs, keeping account of the weights in the definition of $d_{\sigma,[t_2,t_1]}$. It suffices to show that each of the following quantities is bounded by $\frac1{100}$, uniformly in $k\in2\mathbb N+\sigma$, $t\in[t_2,t_1]$, and $(r,z)\in B_{\sigma,[t_2,t_1]}(10)$, upon choosing $T$ sufficiently small (depending on $t_1,\lambda$): for $f_k$,
\eqn{
c_f\coloneqq \frac{w_k^r(t)}{w_k^z(t)}|\dd_{z_k}f_k| + \frac{w_k^r(t)}{w_{k-2}^r}|\dd_{b_k}f_k||\dd_{r_{k-2}}b_k| + \frac{w_k^r(t)}{w_{k-2}^z(t)}|\dd_{b_k}f_k||\dd_{z_{k-2}}b_k|,
}
for $g_k$,
\eqn{
c_g &\coloneqq w_k^r(t)\int_t^{t_1}\Bigg(\frac{|\dd_{z_k}g_k(t')|}{w_k^z(t')}+\frac{|\dd_{b_k}g_k(t')||\dd_{r_{k-2}}b_k(t')|}{w_{k-2}^r(t')}+\frac{|\dd_{b_k}g_k(t')||\dd_{z_{k-2}}b_k(t')|}{w_{k-2}^z(t')}\\
&\quad\quad\quad\quad\quad\quad\quad\quad+\frac{|\dd_{u_{k+1}}g_k(t')||\dd_{r_{k+2}}u_{k+1}(t')|}{w_{k+2}^r(t')}+\frac{|\dd_{u_{k+1}}g_k(t')||\dd_{z_{k+2}}u_{k+1}(t')|}{w_{k+2}^z(t')}\Bigg)dt',
}
and for $h_k$,
\eqn{
c_h &\coloneqq w_k^z(t)\int_t^{t_1}\Bigg(\frac{|\dd_{r_k}h_k(t')|}{w_k^r(t')}+\frac{|\dd_{z_k}h_k(t')|}{w_k^z(t')}+\frac{|\dd_{b_k}h_k(t')||\dd_{r_{k-2}}b_k(t')|}{w_{k-2}^r(t')}\\
&\quad\quad\quad\quad\quad\quad\quad\quad\quad+\frac{|\dd_{b_k}h_k(t')||\dd_{z_{k-2}}b_k(t')|}{w_{k-2}^z(t')}+\frac{|\dd_{u_{k+1}}h_k(t')||\dd_{r_{k+2}}u_{k+1}(t')|}{w_{k+2}^r(t')}\\
&\quad\quad\quad\quad\quad\quad\quad\quad\quad+\frac{|\dd_{u_{k+1}}h_k(t')||\dd_{z_{k+2}}u_{k+1}(t')|}{w_{k+2}^z(t')}\Bigg)dt'.
}
We write partial derivatives of $b_k$ and $u_{k+1}$ in the sense that we can express $b_k=b_k(r_{k-2},z_{k-2})$ and $u_{k+1}=u_{k+1}(r_{k+2},z_{k+2})$ as in the relations \eqref{ufromrandz}:
\eqn{
b_k=N_{k-2}^\alpha\delta_{k-2}r_{k-2}R_3(z_{k-2})-N_{k-1}^2
}
and
\eqn{
u_{k+1}=\delta_{k+2}(1+R_1(z_{k+2})r_{k+2}).
}
We will index the individual terms as $c_f=c_{f,1}+\cdots+c_{f,3}$, $c_g=c_{g,1}+\cdots+c_{g,5}$, etc.

The sum $c_f+c_g+c_h$ controls the expansion constant for the map $(F,G)$ on $B_{\sigma,[t_2,t_1]}(10)$. There is a single exception when $\sigma=k=2$ and $b_2$ is not a function of $(r_0,z_0)$ but rather of $u_1$, through \eqref{u0duhamelIII}. In this case the analysis still carries through, with $|b_2-\tilde b_2|$ controlled by $|r_2-\tilde r_2|$ and $|z_2-\tilde r_2|$.

We claim that using the same elements as above, one can compute
\eqn{
c_f+c_g+c_h\lesssim\lambda^{O(1)}\frac{T}{t_1}+\lambda^{-4+\alpha}.
}
Thus the maps $(F,G)$ can be arranged to be a contraction on $B_{\sigma,[t_2,t_1]}(10)$ by taking $\lambda$ sufficiently large, then $T$ sufficiently small depending on $\lambda$, $t_1$. We show the details for a few terms; the rest are analogous. First, we point out that there are many terms that are straightforward to estimate because they contain a factor of $T$ and an exponentially decaying factor in time. All three terms in $c_f$ are of this form. For instance, the first is
\eqn{
c_{f,1}&=\frac{w_k^r(t)}{w_k^z(t)}|\dd_{z_k}f_k|\\
&\lesssim_\lambda\frac{\exp(2N_k^2t)}{\exp(N_k^2t)}T\|b_k\|_{L^\infty}|R_1(z_k)||R_1'(z_k)|\exp(T\|b_k\|_{L^\infty}R_1^2(z_k)-2N_k^2t_1)\\
&\lesssim_\lambda TN_k^2\exp(-N_k^2(2t_1-O(T\lambda^{-4+\alpha})))\\
&\lesssim T/t_1
}
by \eqref{Rrationalfunctionbounds}, \eqref{bbound}, \eqref{tauconstraint}, and the fact that $(r,z)\in B_{\sigma,[t_1,t_0]}(\eps)\cap B_{\sigma,[t_2,t_0]}(10)$. The key point in the last line is that $\max_{N>0}N^2\exp(-tN^2)\lesssim t^{-1}$. Throughout we have neglected factors of $\lambda$, including those appearing in quantities such as $\exp(\tau_kN_k^2)$. These can all be defeated by the choice of $T$.

The same strategy can be used for terms that contain the fast decay coming from $u_{k+1}$. For illustration, consider the last term in $c_h$. We have
\eqn{
c_{h,6}&=w_k^z(t)\int_t^{t_1}\frac{|\dd_{u_{k+1}}h_k(t')||\dd_{z_{k+2}}u_{k+1}(t')|}{w_{k+2}^z(t')}dt'\\
&\quad\lesssim_\lambda\exp(N_k^2t)\int_t^{t_1}\frac{\exp(N_k^2(t'-t))N_k^\alpha\delta_k^{-1}|R_5(z_k)|u_{k+1}\delta_{k+2}|R_1'(z_{k+2})|}{\exp(N_{k+2}^2t')}dt'\\
&\quad\lesssim_\lambda N_k^2\int_t^{t_1}\exp(-2(N_{k+2}^2+N_k^2)t')dt'\\
&\quad\lesssim N_k^2T\exp(-2N_{k+2}^2t)\lesssim T/t_1
}
using all the same elements, as well as $t\geq t_1/2$.

Having now given the ingredients for estimating the bulk of the terms in $c_f+c_g+c_h$, let us show the details for two terms that require somewhat more sensitive treatment. The first is the contribution to $c_{g,1}$, from when $\dd_{z_k}$ falls on the $b_kR_2(z_k)$ term in $g_k$:
\eqn{
c_{g,1,2}=w_k^r(t)\int_t^{t_1}w_k^z(t')^{-1}\exp(\int_t^{t'}b_kR_1^2(z_k)dt'')b_kR_2'(z_k)dt',
}
which we name based on its position within $c_{g,1}$. Recall that $w_k^r$ is defined piecewise, and we must separately consider the two cases. For $k$ sufficiently small, $w_k^r=\lambda^{\frac12(4-\alpha)}$ and we estimate
\eqn{
c_{g,1,2}&\lesssim_\lambda N_k^2\int_t^{t_1}\exp(-N_k^2t'+O((t'-t)\lambda^{-4+\alpha}N_k^2))\lesssim N_k^2T\exp(-N_k^2t)
}
using \eqref{Rrationalfunctionbounds} for $R_2'$. The smallness follows as above. On the other hand, suppose $k$ is large so that $w_k^r=\exp(2N_k^2(t-\tau_k))$. Then
\eqn{
c_{g,1,2}&\lesssim \|b\|_{L^\infty}\int_{t}^{t_1}\exp(-N_k^2(t'-t)(2-O(\lambda^{-4+\alpha})))dt'\\
&\lesssim N_k^{-2}\|b\|_\infty(1-\exp(-N_k^2t_1))\\
&\lesssim\lambda^{-4+\alpha}.
}

Finally, let us consider the contribution to $c_{g,2}$ from when $\dd_{b_k}$ falls on the $b_kR_2(z_k)$ term:
\eqn{
c_{g,2,2}&=w_k^r(t)\int_t^{t_1}\frac{|\dd_{r_{k-2}}b_k(t')|}{w_{k-2}^r(t')}|R_2(z_k)|\exp(\int_t^{t'}b_kR_1^2(z_k)dt'')dt'.
}
Suppose first that $w_k^r(t)=\lambda^{\frac12(4-\alpha)}$. Then $\exp(N_k^2t)\leq\lambda^{\frac34\alpha-1}$ and we can estimate
\eqn{
c_{g,2,2}&\lesssim_\lambda N_k^2\int_t^{t_1}\exp((t'-t)\|b_k\|_{L^\infty})dt'\leq TN_k^2\exp(N_k^2t)\lesssim_\lambda T/t_1.
}
If instead $\exp(N_k^2t)\geq\lambda^{\frac34\alpha-1}$, then
\eqn{
c_{g,2,2}&\lesssim \lambda^{-\frac12(4-\alpha)}N_{k-2}^\alpha\delta_{k-2}\exp(2N_k^2t)\int_t^{t_1}|R_3(z_{k-2})\|\exp(-2N_k^2t'+(t'-t)\|b_k\|_{L^\infty})dt'\\
&\lesssim \lambda^{-\frac32(4-\alpha)}N_k^2\int_t^{t_1}\exp(-N_k^2(t'-t)(2-O(\lambda^{-4+\alpha})))dt'\\
&\lesssim\lambda^{-\frac32(4-\alpha)}(1-\exp(-N_k^2T))
}
using the same collection of tools.

Having proved that $(F,G)$ is a contraction on $B_{\sigma,[t_2,t_1]}$, the existence and bounds entailed by $\mathcal H(t_2)$ are immediate.
\end{proof}

There is one final technical ingredient needed to close the bootstrap.

\begin{prop}\label{closedhypothesisprop}
    The set of $t$ for which $\mathcal H'(t)$ holds is closed in $[0,t_0]$.
\end{prop}

\begin{proof}
    Assume $\mathcal H'(t_i)$ holds for a sequence $t_i\to t_\infty$ in $[0,t_0]$; we claim $\mathcal H'(t_\infty)$ as well. Note that $\mathcal H'(s)$ implies $\mathcal H'(t)$ if $s\leq t$. Thus we may reduce to the case where $t_i$ is decreasing. Moreover, by backward uniqueness (Proposition~\ref{localprop}), if $t_i<t_j$, the solutions provided by $\mathcal H'(t_i)$ and $\mathcal H'(t_j)$ agree on $[t_j,\infty)$. Thus we arrive at a solution on the interval $(t_\infty,\infty)$ obeying $(r,z)\in B_{\sigma,(t_\infty,t_0]}(\eps)$. Inserting these bounds into \eqref{rzsystem} implies that $\dot r_k,\,\dot z_k$ are uniformly (in $t$) bounded near $t_\infty$ for each $k$. Those estimates can be bootstrapped\footnote{In the other sense of the term.} again using the derivative of \eqref{rzsystem} to obtain uniform (in $t$) bounds on $\ddot r_k,\,\ddot z_k$. Thus there exist extensions of $r_k$ and $z_k$ in the space $C^1([t_\infty,\infty))$, which implies $\mathcal H'(t_\infty)$.
\end{proof}

Putting these pieces together, we arrive at the following proposition. It will be convenient going forward to parameterize the space of solutions by $\mu=(\mu_1,\mu_2,\ldots)$ chosen from
$$M\coloneqq[1/10,10]^{\mathbb N+1}\subset X^s\quad(s\leq0).$$

\begin{prop}\label{bootstrapconclusionprop}
    For any $\mu\in M$, there exist $u,v$ such that the following hold:
    \begin{enumerate}[(i)]
        \item $u,v$ are non-negative solutions of \eqref{system} on $[0,\infty)$.
        \item When $u$ is written in $r,z$ coordinates, the data obey
    \eq{
    r_k(0)\in[\frac9{10},\frac{11}{10}],\quad z_k(0)\in\lambda^{\alpha-2}\mu_k^{-1}[\frac25,\frac35]\label{datarange}
    }
    for all $k\in2\mathbb N+1$.
    \item When $v$ is written in $r,z$ coordinates, the data obey \eqref{datarange} for all $k\in2\mathbb N+2$.
    \item Both solutions possess the instantaneous smoothing property
    \eq{\label{exponentialdecay}
    u_k(t)+v_k(t)\lesssim_\lambda \delta_ke^{-N_k^2t}
    }
    for all $k\in\mathbb N$, $t\geq0$. More precisely, $u$ and $v$ obey \eqref{manifoldboundsk-1}--\eqref{manifoldboundsk} when $k\in2\mathbb N+1$ and $k\in2\mathbb N+2$, respectively.
    \item\label{samedataforzeromode} The initial data agree at the zeroth mode,
    \eqn{
    u_0(0)=v_0(0).
    }
    \end{enumerate}
\end{prop}

\begin{proof}
   The strategy is as depicted in the first part of Figure~\ref{Tdefinitionfigure}. First we apply Proposition~\ref{stablemanifoldprop} with $\tau_k=N_k^{-2}\log(\lambda^{\alpha-2}/\mu_k)$ $(k\in2\mathbb N+1$) and $\sigma=1$ to obtain the solution $u$ on $[t_0,\infty)$. It is easy to verify that if $\mu\in M$, then this choice of $\tau_k$ satisfies \eqref{tauconstraint}. Next we propagate the existence and bounds for $u$ backward in time to $t=0$. It suffices to prove $\mathcal H'(0)$ with $\sigma=1$ which we do using the general bootstrap recipe from Proposition~1.21 in~\cite{taobook}: Propositions~\ref{bootsrapboundsprop} and \ref{localprop} together imply that $\{t\in[0,t_0]:\mathcal H'(t)\}$ is open in $[0,t_0]$. Proposition~\ref{closedhypothesisprop} implies that it is closed. Finally, Lemma~\ref{utorzlemma} implies that it contains $t_0$. We conclude $\mathcal H'(0)$ which guarantees the existence of $u$ on $[0,\infty)$ as well as the estimates \eqref{datarange} and \eqref{exponentialdecay}.
   
   Now that we have constructed $u$, we can begin constructing $v$ by repeating a parallel procedure with $\sigma=2$. First we apply Proposition~\ref{stablemanifoldprop} with $\tau_k=N_k^{-2}\log(\lambda^{\alpha-2}/\mu_k)$ ($k\in2\mathbb N+2$) with the additional input $a=u_0(0)$. Note that this $a$ satisfies the requirement \eqref{arange} by a simple calculation involving \eqref{ufromrandz} and \eqref{datarange} for $u_0$. A bootstrap argument identical to the one for $u$ applies to produce $v$ on $[0,\infty)$ obeying \eqref{datarange} and \eqref{exponentialdecay}. Because of the choice $a=u_0(0)$ in the application of Proposition~\ref{stablemanifoldprop}, we have (\ref{samedataforzeromode}).
\end{proof}

We remark that the order of operations in the above proof is quite important because the initial value of $u_0$ must be known to construct $v$; see Figure~\ref{Tdefinitionfigure}.

\section{Initial data and lack of uniqueness}\label{constructthedataprovethetheoremsection}

Proposition~\ref{bootstrapconclusionprop} by itself does not prove non-uniqueness because the initial data of $u$ and $v$ agree only at the zeroth mode. The final step of the argument is to show topologically that there exists a choice of $\mu$ to make the data fully coincide.

First we make a brief comment on notation. Recall from Section~\ref{bootstrapsection} that $u$ is naturally expressed in the coordinates $(r_k,z_k)_{k\in2\mathbb N+1}$ and $v$ in $(r_k,z_k)_{k\in2\mathbb N+2}$. Thus the pair $u,v$ is determined precisely by $(r_k,z_k)_{k\geq1}$. We write $r^0$ and $z^0$ for the initial data $(r_k(0))_{k\geq1}$ and $(z_k(0))_{k\geq1}$, while $u^0$ and $v^0$ refer to the data $(u_k(0))_{k\geq0}$ and $(v_k(0))_{k\geq0}$.

\begin{proof}[Proof of Theorem~\ref{maintheorem}]

The initial data is constructed as follows (see Figure~\ref{Tdefinitionfigure}). Proposition~\ref{bootstrapconclusionprop} defines a mapping $\mu\mapsto (u,\,v)$ to solutions of \eqref{system} for $\mu\in M$. Consider in particular the mapping to the initial data,
\eqn{
\mu\mapsto(u^0,v^0),\quad \mu\in M.
}
We claim there exists a $\mu\in M$ such that $u^0(\mu)=v^0(\mu)$.

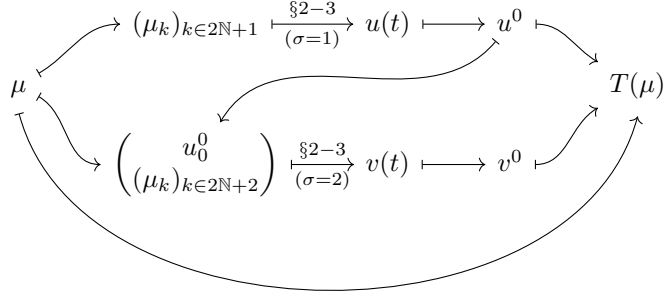
\begin{figure}
    \centering
    \begin{tikzcd}[row sep=.2cm]
    & (\mu_k)_{k \in 2\mathbb N + 1} \arrow[r, maps to, "\S\ref{stablemanifoldsection}-\ref{bootstrapsection}", "(\sigma=1)"']  & u(t) \arrow[r, maps to]  & u^0 \arrow[dr, maps to, out=0, in=150] \arrow[ddll,maps to, out=230, in=60] & \\
    \mu \arrow[rrrr, mapsto, start anchor={[xshift=-1.7ex, yshift=-2.5ex]east}, end anchor={[xshift=0ex, yshift=-.7ex]south}, out=-75, in=255] \arrow[ur, maps to, out=30, in=180] \arrow[dr, maps to, out=-30, in=180] & & & & T(\mu) \\
    & \left(\begin{matrix}{u^0_0}\\(\mu_k)_{k \in 2\mathbb N + 2}\end{matrix}\right) \arrow[r, maps to, "\S\ref{stablemanifoldsection}-\ref{bootstrapsection}", "(\sigma=2)"'] & v(t) \arrow[r, maps to] & v^0 \arrow[ur, maps to, out=0, in=210] & 
\end{tikzcd}
    \caption{A schematic view of the definition of $T(\mu)$. The values of $\mu_k$ for $k$ odd determine $u^0$, while the values of $\mu_k$ for $k$ even along with $u^0_0$ determine $v^0$. Fixed points of $T$ correspond to the  desired outcome $u^0=v^0$.}\label{Tdefinitionfigure}
\end{figure}

Consider the map $T:M\to \mathbb R^{\mathbb N+1}$ given by
\begin{equation*}
T(\mu)_k=\mu_k\begin{cases}
    v_k^0(\mu)/u_k^0(\mu),&k\text{ odd}\\
    u_k^0(\mu)/v_k^0(\mu),&k\text{ even}
\end{cases},\quad k\geq1.
\end{equation*}
We claim that $T:M\to M$. Indeed, suppose $k\in2\mathbb N+1$. Then \eqref{ufromrandz} and \eqref{datarange} imply $v_k^0(\mu)\in[\delta_{k+1},3\delta_{k+1}]$ and $u_k^0(\mu)/\mu_k\in[\delta_{k+1},6\delta_{k+1}]$. Dividing the two, we see that indeed $T(\mu)_k\in[1/10,10]$. If instead $k\in2\mathbb N+2$, the same argument applies with $u$ and $v$ exchanged.

Endowing $M$ with the $\|\cdot\|_s$ norm for some $s\in(-2,0)$, it becomes a compact subset of $X^s$. We claim that $T:M\to M$ is continuous. For any $k\geq1$, by \eqref{ufromrandz} and the fact that $\delta_{k+1}/\delta_k=\lambda^{-\alpha+2}$, one computes that
\eqn{
T(\mu)_k=\lambda^{-\alpha+2}\mu_kf_k(r^0(\mu),z^0(\mu))
}
where $f:(\mathbb R^\mathbb N)^2\to\mathbb R^\mathbb N$ is defined as
\eq{\label{fdefinition}
        f_k(\rho,\zeta)\coloneqq\frac{1+\rho_{k+1}R_1(\zeta_{k+1})}{\rho_kR_3(\zeta_k)},
        }
recalling \eqref{rationalfunctions}. Let us slightly abuse notation by writing $f(\mu)$ for $f(r^0(\mu),z^0(\mu))$. Based on the inequality
$$\|\mu f(\mu)-\mu'f(\mu')\|_s\leq\|\mu-\mu'\|_s\|f(\mu)\|_0+\|\mu'\|_0\|f(\mu)-f(\mu')\|_{s}$$
it suffices to show that $\mu\mapsto f(r^0(\mu),z^0(\mu))$ is bounded $M\to X^0$ and continuous $M\to X^s$. The boundedness is immediate from \eqref{datarange} and Lemma~\ref{topologylemma}(i). To show continuity, note that Lemma~\ref{topologylemma}(ii)-(iii) together imply that $\mu\to(r^0,z^0)$ is continuous as a mapping $M\to A\subset(X^s)^2$. We then conclude by part (i) of the lemma. Since $T$ is a continuous map on the compact, convex subset $M\subset X^s$, Schauder's theorem provides a fixed point $\mu_*\in M$. One immediately sees from the definition of $T$ and Proposition~\ref{bootstrapconclusionprop}(\ref{samedataforzeromode}) that $\mu_*$ gives rise to solutions $u,v$ of \eqref{system} with the same data $u^0$.

To see that $u$ and $v$ are distinct, consider the asymptotics proved in Proposition~\ref{stablemanifoldprop}. For all $t\geq t_0$, we have that $v_1\lesssim\delta_2\exp(-2N_2^2(t-\tau_2))$, while $u_1$ is comparable to $\delta_1\exp(-N_1^2(t-\tau_1))$. These are clearly incompatible when $t$ is sufficiently large. To see that $u$ and $v$ are Leray--Hopf, i.e., satisfy \eqref{energyinequality}, consider the energy balance for the truncated system,
\eqn{
\frac d{dt}\sum_{k=0}^m\frac{u_k^2}2+\sum_{k=0}^mN_k^2u_k^2=-N_m^\alpha u_mu_{m+1}^2
}
which implies
\eqn{
\sum_{k=0}^m\frac{u_k^2(t)}2+\int_0^t\sum_{k=0}^mN_k^2u_k^2(s)ds\leq\sum_{k=0}^m\frac{(u_k^0)^2}2
}
for all $m$ (recalling $u_k\geq0$). Then \eqref{energyinequality} follows by Fatou's lemma.
\end{proof}

\appendix

%\section{Basic existence theory for the Obukhov model}

%First we prove local existence of regular solutions from subcritical data, analogous to the classical theory of Fujita-Kato for the Navier--Stokes equations.

%\begin{thm}
%%    For any $s>\alpha-2$ and initial data $u^0\in H^s$, there exists $c>0$ such that there exists a unique solution $u$ of \eqref{system} in the space $L^\infty([0,T],H^s)$, where $T=c/\|u_0\|^{...}$. Moreover, the solution is smooth for positive times in the sense that 
%\end{thm}

\section{Stability in the delays}\label{stabilityappendixsection}

\begin{lem}\label{gronwalllemma}
We have the following Gr\"onwall-type inequalities:
\begin{enumerate}[(i)]
    \item Assume that $g(t)\in (C\cap L^1)([t_0,\infty);[0,\infty))$ obeys
    \eq{
    g(t)\leq C\left(\delta e^{-t}+t^{-1}\int_{t_0}^tg(t')dt'+
    \frac{t-t_0}{t}\int_t^\infty(t')^{-1}g(t')dt'\right)\label{gronwallhypothesis}
    }
    for all $t\in[t_0,\infty)$. Then
    \eqn{
    g(t)\lesssim_{C,t_0,t} \delta.
    }
    \item Assume that $g(t)\in C((0,t_0];[0,1])$ obeys
    \eqn{
    g(t)\leq C(g(t_0)^{t/t_0}+\frac1t\int_t^{t_0}g(t')dt').
    }
    Then
    \eqn{
    g(t)\lesssim Cg(t_0)^{t/t_0}(1+(t_0/t)^{C+1}).
    }
\end{enumerate}
    
\end{lem}

\begin{proof}
    Consider the function
    \eqn{
    G(t)=\left(\frac1{t_0}\int_{t_0}^tg+\int_t^\infty\frac{g(t')}{t'}\right)t^Ce^{-Ct/t_0}.
    }
    Differentiating and applying \eqref{gronwallhypothesis},
    \eqn{
    G'(t)&=\left((t_0^{-1}-t^{-1})g-Ct_0^{-1}(t_0^{-1}-t^{-1})\int_{t_0}^tg-C(t_0^{-1}-t^{-1})\int_t^\infty\frac{g(t')}{t'}\right)t^Ce^{-Ct/t_0}\\
    &\leq C\left((t_0^{-1}-t^{-1})\delta e^{-t}-(t_0^{-1}-t^{-1})^2\int_{t_0}^tg-(1-t_0/t)\int_t^\infty\frac{g(t')}{t'}\right)t^Ce^{-Ct/t_0}\\
    &\leq C\delta(t_0^{-1}-t^{-1}) t^Ce^{-(1+C/t_0)t}.
    }
    Moreover from the assumption that $g$ is integrable, we have $G(t)\to0$ as $t\to\infty$. Thus by the fundamental theorem of calculus,
    \eqn{
    G(t)&\lesssim_{C,t_0}\delta
    }
    for $t\geq t_0$. Then by \eqref{gronwallhypothesis},
    \eqn{
    g(t)&\lesssim_C\delta e^{-t}+\max(\frac{t_0}t,1-\frac{t_0}t)t^{-C}e^{Ct/t_0}G(t)
    }
    for all $t\geq t_0$. Using the upper bound for $G$, we obtain the result.
    
    For (ii), we define $f(t)=tg(t)$ which obeys
    \eqn{
    f(t)\leq C(tg(t_0)^{t/t_0}+\int_t^{t_0}\frac{f(t')}{t'}dt').
    }
    Then by the standard Gr\"onwall inequality (backward in time), we have
    \eqn{
    f(t)&\leq Ctg(t_0)^{t/t_0}+C^2\int_t^{t_0}g(t_0)^{t'/t_0}\left(\frac{t'}t\right)^Cdt'\\
    &\lesssim Ctg(t_0)^{t/t_0}(1+\left(\frac{t_0}t\right)^{C+1})
    }
    which proves the claim.
\end{proof}

In the next lemma we continue with the notation of Section~\ref{constructthedataprovethetheoremsection}. In particular, $(r,z)$ should be taken to contain the $r_k,z_k$ coordinate expressions for both $u$ ($k$ odd) and $v$ ($k$ even).

\begin{lem}\label{topologylemma}
    Let $s\in(-2,0)$. Define $A\subset (X^s)^2$ to be the subset of $(r^0,z^0)$ satisfying \eqref{datarange}. We have the following technical facts:
    \begin{enumerate}[(i)]
        \item The function $f:A\to X^s$ defined in \eqref{fdefinition} is continuous. Moreover it is bounded as a map $A\to X^0$.
        \item For every $t>0$, the mapping $\mu\mapsto(r(t),z(t))$ defined by Proposition~\ref{bootstrapconclusionprop} is in $C(M;(X^s)^2)$.
        \item We have the convergence $(r(t),z(t))\to(r^0,z^0)$ in $(X^s)^2$ as $t\to0+$ uniformly in $\mu\in M$.
    \end{enumerate}
\end{lem}

Clearly (ii) and (iii) together imply that $\mu\mapsto(r^0,z^0)$ is in $C(M;A)$, using \eqref{datarange} to narrow the range to $A$.

\begin{proof}
    It is easy to see that the partial derivatives of $f_k$ are bounded uniformly in $k$ as a result of \eqref{datarange}. This implies the continuity claim in $(i)$. The boundedness claim is immediate from \eqref{datarange}.

    For (ii), we must first show that $u(t_0)$ and $v(t_0)$ depend continuously on $\mu$. As in Sections~\ref{stablemanifoldsection}--\ref{bootstrapsection}, we let the variable ``$u$'' stand in for $u$ or $v$ and distinguish them by the parity $\sigma$. Fix $\sigma\in\{1,2\}$ and $\mu,\mu'\in M$. Let $w^\sigma$ and $w^{\sigma+1}$ be the rescaled differences $((u_k(\mu)-u_k(\mu'))/\delta_k)_{k\in2\mathbb N+\sigma}$ and  $((u_k(\mu)-u_k(\mu'))/\delta_k)_{k\in2\mathbb N+\sigma+1}$, respectively. To estimate the differences we make use of the Duhamel formula from the construction in Proposition~\ref{stablemanifoldprop},
    \eqn{
    u_{k-1}&=N_{k-1}^\alpha\int_t^\infty\exp\Big(N_{k-1}^2(t'-t)-\int_{t}^{t'}N_{k-2}^\alpha u_{k-2}\Big)u_{k}^2(t')dt'\\
u_k&=\delta_k e^{-N_k^2(t-\tau_k)}+\int_{t_0}^t\exp(-N_k^2(t-t'))(N_{k-1}^\alpha u_{k-1}u_k-N_k^\alpha u_{k+1}^2)(t')dt'
    }
    for all $k\in2\mathbb N+\sigma$. Using the first formula one obtains an expression for $u_{k-1}(\mu)-u_{k-1}(\mu')$ from which we can estimate, for any $s<0$,
    \eqn{
    \|w^{\sigma+1}(t)\|_s&\lesssim_\lambda \sup_kN_k^\alpha\int_t^\infty\exp(N_{k-1}^2(t'-t))\Big(\int_t^{t'}N_k^\alpha \|w^\sigma(t'')\|_sdt''N_k^{-2\alpha+4}\exp(-2N_{k}^2t')\\
    &\quad+N_k^{-\alpha+2}\|w^\sigma(t')\|_s\exp(-N_k^2t')\Big)dt'\\
    &\lesssim_\lambda \sup_k\Big(N_k^4\int_t^\infty\int_t^{t'}\exp(N_{k-1}^2(t'-t)-2N_k^2t')\|w^\sigma(t'')\|_sdt''dt'\\
    &\quad+N_k^2\int_t^\infty\exp(N_{k-1}^2(t'-t)-N_k^2t')\|w^\sigma(t')\|_sdt'\Big)\\
    &\lesssim_\lambda\int_t^\infty\int_t^{t'}(t')^{-2}\|w^\sigma(t'')\|_sdt''dt'+\int_t^\infty(t')^{-1}\|w^\sigma(t')\|_sdt'.
    }
    Here we have used the estimate \eqref{manifoldboundsk} for the extra $u_{k-2}$ and $u_k$ factors that appear, along with the basic fact that $\sup_k N_k^pe^{-tN_k^2}\lesssim t^{-\frac p2}$. At this point we have begun neglecting the distinction between $N_k$, $N_{k-1}$, etc.\ since the implicit constants may depend on $\lambda$.
    
    Applying Tonelli's theorem to the first term yields
    \eq{
    \|w^{\sigma+1}(t)\|_s\lesssim_\lambda\int_t^\infty(t')^{-1}\|w^\sigma(t')\|_sdt'.\label{wsigmaplusonebound}
    }
    Similarly from the Duhamel formula, \eqref{manifoldboundsk-1}, and \eqref{manifoldboundsk},
    \eqn{
    \|w^\sigma(t)\|_s&\lesssim_\lambda \sup_k\Big(\|\mu-\mu'\|_se^{-N_k^2t}+N_k^2\int_{t_0}^t\exp(-N_k^2(t-t'))\exp(-N_k^2t')\|w(t')\|_sdt'\Big)\\
    &\lesssim_\lambda N_0^{-\alpha+2}\|\mu-\mu'\|_se^{-N_0^2t}+t^{-1}\int_{t_0}^t(\|w^\sigma(t')\|_s+\|w^{\sigma+1}(t')\|_s)dt'.
    }
    Combining this with \eqref{wsigmaplusonebound},
    \eqn{
    \|w^\sigma(t)\|_s&\lesssim_\lambda N_0^{-\alpha+2}\|\mu-\mu'\|_se^{-N_0^2t}+t^{-1}\int_{t_0}^t\|w^\sigma(t')\|_sdt'+t^{-1}\int_{t_0}^t\int_{t'}^\infty \frac{\|w^\sigma(t'')\|_s}{t''}dt''dt'.
    }
    We can transform the third term using Tonelli's theorem and split the $dt''$ integral into $[t_0,t]$ and $[t,\infty)$ parts to obtain
    \eq{\label{wsigmabound}
    \|w(t)\|_s&\lesssim_\lambda N_0^{-\alpha+2}\|\mu-\mu'\|_se^{-N_0^2t}+t^{-1}\int_{t_0}^t\|w^\sigma(t')\|_sdt'+\frac{t-t_0}{t}\int_t^\infty\frac{\|w^\sigma(t')\|_s}{t'}dt'.
    }
    Note that we also have
    \eq{\label{wapriori}
    \|w^\sigma(t)\|_s\lesssim_\lambda N_0^{-\alpha+2}\exp(-tN_0^2)
    }
    immediately from the triangle inequality, \eqref{manifoldboundsk} and \eqref{tauconstraint}.
    With \eqref{wsigmabound} and \eqref{wapriori} in hand, rescaling away the factors of $N_0$, $g(t)=\|w^\sigma(t)\|_s$ obeys the hypotheses of Lemma~\ref{gronwalllemma} (in particular with \eqref{wapriori} providing the integrability) and we obtain
    \eq{\label{pointwisew}
    \|w^\sigma(t)\|_s\lesssim_{\lambda,t} N_0^{-\alpha+2}\|\mu-\mu'\|_s
    }
    for all $t\geq t_0$. To transfer this to $w^{\sigma+1}$, consider \eqref{wsigmaplusonebound} evaluated at $t=t_0$. By \eqref{wapriori}, the integrand is dominated by an integrable function uniformly in $\mu$. Thus along any sequence $\mu_k\to\mu$, $\|w^\sigma(t)\|_s\to0$ pointwise in time which implies $\|w^{\sigma+1}(t_0)\|_s\to0$ by dominated convergence. This proves that $\mu\mapsto u(t_0)/\delta\in C(M; X^{s})$ (or that $\mu\mapsto u(t_0)\in C(M;X^{s+\alpha-2})$). 

    Next we propagate the continuity backward to all positive times. For any $\sigma\in\{1,2\}$, $0<t\leq t_0$, and $k\in2\mathbb N+\sigma$ we can write
    \eqn{
    u_{k-1}(t)&=\exp(N_{k-1}^2(t_0-t))u_{k-1}(t_0)+N_{k-1}^\alpha\int_t^{t_0}\exp(-\int_t^{t'}b_k)u_k^2(t')dt'\\
    u_k(t)&=\exp(N_k^2(t_0-t))u_k(t_0)+\int_t^{t_0}\exp(N_k^2(t'-t))(N_{k-1}^\alpha u_{k-1}u_k-N_k^\alpha u_{k+1}^2)dt'.
    }
    With $u,u'\in X^s$ coming from $\mu$ and $\mu'$ respectively, let us define $w_j=(u_j-u_j')/\delta_j$ for all $j\geq0$. We have
    \eqn{
    |w_{k-1}(t)|&\lesssim_\lambda\exp(N_{k-1}^2(t_0-t))\min(|w_{k-1}(t_0)|,\exp(-2N_k^2t_0))\\
    &\quad+N_k^4\int_t^{t_0}\exp(N_{k-1}^2(t'-t)-2N_k^2(t'))\int_t^{t'}|w_{k-2}(t'')|dt''dt'\\
    &\quad+N_k^2\int_t^{t_0}\exp(N_{k-1}^2(t'-t)-N_k^2t')|w_k(t')|dt'.
    }
    Note that the first term is bounded from above by $|w_{k-1}(t_0)|^\frac12$. Thus
    \eqn{
    \sup_{k}\lambda^{sk}|w_{k-1}(t)|&\lesssim_\lambda \|w(t)\|_s^\frac12+\int_t^{t_0}\int_t^{t'}(t')^{-2}\|w(t'')\|_sdt''dt'+\int_t^{t_0}(t')^{-1}\|w(t')\|_sdt'.
    }
    One carries out a similar computation for $w_k$ and adds the results, once again applying Tonelli's theorem, to find
    \eqn{
    \|w(t)\|_s&\lesssim_\lambda\|w(t_0)\|_s^{t/t_0}+\frac1t\int_t^{t_0}\|w(t')\|_sdt'.
    }
    Finally, Lemma~\ref{gronwalllemma}(ii) lets us conclude continuity of the mapping $u(t_0)/\delta\mapsto u(t)/\delta$ in $X^s$ for each $t>0$.
    
    To summarize, we have shown that $\mu\mapsto u(t)/\delta\in C(M;X^s)$ for all $t>0$. At this point we transfer back to $(r,z)$ coordinates to conclude the proof of (ii). We claim that the mapping $u(t)/\delta\to(r,z)(t)$ is in $C(X^s; (X^s)^2)$. The continuity for $r$ is trivial using \eqref{rzdefinition} and the bounds coming from $(r,z)\in B_{\sigma,[0,t_0]}(10)$. For $z$, one can differentiate \eqref{rzdefinition} to find that if $z$ and $z'$ arise from distinct $\mu$'s, then
    \eqn{
    |z_k-z_k'|\leq\left(\frac{|u_{k-1}-u_{k-1}'|}{\delta_k}+|r_k-r_k'|\right)\sup_{\mu\in M}\frac{z_k}{r_k}+\frac{|u_{k}-u_{k}'|}{\delta_k}\sup_{\mu\in M}\left(\frac{\delta_kz_k}{u_k}\right)
    }
    where we have used the definition of $z_k$ to simplify the second supremum. We claim that both the suprema are bounded for each $t$ uniformly in $k$. For the first, it is immediate from $\mathcal H'(0)$ and \eqref{tauconstraint} that $r_k\gtrsim1$ while $z_k=O_\lambda(1)$. For the second, we use the fact from \eqref{ufromrandz} that $\delta_kz_k/u_k=r_k^{-1}z_k/R_3(z_k)=(2r_k)^{-1}(z_k^2+1)=O_\lambda(1)$ from which we conclude.

    For (iii), observe directly from the system for $r,z$ and the fact that they are in $B_{\sigma,[0,t_0]}(\eps)$ that
    \eqn{
    |\dot r_k|+|\dot z_k|\lesssim_\lambda N_k^2 e^{-N_k^2t}
    }
    for all $t\in[0,t_0]$. It follows that
    \eqn{
    \|r(t)-r^0\|_s+\|z(t)-z^0\|_s\lesssim_\lambda\sup_{k\geq1}\lambda^{sk}(1-e^{-N_k^2t})\lesssim_s (N_0^2t)^{-\frac s2}.
    }
    The result follows because $s<0$.
\end{proof}

\bibliographystyle{abbrv}
\bibliography{literature}

\begin{thebibliography}{10}

\bibitem{abc}
D.~Albritton, E.~Bru\'{e}, and M.~Colombo.
\newblock Non-uniqueness of {L}eray solutions of the forced {N}avier-{S}tokes equations.
\newblock {\em Ann. of Math. (2)}, 196(1):415--455, 2022.

\bibitem{barbatodyadic2011}
D.~Barbato, F.~Morandin, and M.~Romito.
\newblock Smooth solutions for the dyadic model.
\newblock {\em Nonlinearity}, 24(11):3083--3097, 2011.

\bibitem{bardossymmetrybreaking}
C.~Bardos, M.~C. Lopes~Filho, D.~Niu, H.~J. Nussenzveig~Lopes, and E.~S. Titi.
\newblock Stability of two-dimensional viscous incompressible flows under three-dimensional perturbations and inviscid symmetry breaking.
\newblock {\em SIAM J. Math. Anal.}, 45(3):1871--1885, 2013.

\bibitem{barkerprangejinsymmetrybreaking}
T.~Barker, C.~Prange, and J.~Tan.
\newblock On symmetry breaking for the {N}avier-{S}tokes equations.
\newblock {\em Comm. Math. Phys.}, 405(2):Paper No. 25, 39, 2024.

\bibitem{buckmastervicol2019review}
T.~Buckmaster and V.~Vicol.
\newblock Convex integration and phenomenologies in turbulence.
\newblock {\em EMS Surv. Math. Sci.}, 6(1-2):173--263, 2019.

\bibitem{buckmastervicol}
T.~Buckmaster and V.~Vicol.
\newblock Nonuniqueness of weak solutions to the {N}avier-{S}tokes equation.
\newblock {\em Ann. of Math. (2)}, 189(1):101--144, 2019.

\bibitem{bulut2023convex}
A.~Bulut, M.~K. Huynh, and S.~Palasek.
\newblock Convex integration above the {O}nsager exponent for the forced {E}uler equations.
\newblock {\em arXiv preprint arXiv:2301.00804}, 2023.

\bibitem{bulut2023non}
A.~Bulut, M.~K. Huynh, and S.~Palasek.
\newblock Non-uniqueness up to the {O}nsager threshold for the forced {SQG} equation.
\newblock {\em arXiv preprint arXiv:2310.12947}, 2023.

\bibitem{cheskidovblowup}
A.~Cheskidov.
\newblock Blow-up in finite time for the dyadic model of the {N}avier-{S}tokes equations.
\newblock {\em Trans. Amer. Math. Soc.}, 360(10):5101--5120, 2008.

\bibitem{dyadicreview}
A.~Cheskidov, M.~Dai, and S.~Friedlander.
\newblock Dyadic models for fluid equations: a survey.
\newblock {\em J. Math. Fluid Mech.}, 25(3):Paper No. 62, 26, 2023.

\bibitem{cfp2007}
A.~Cheskidov, S.~Friedlander, and N.~Pavlovi\'{c}.
\newblock Inviscid dyadic model of turbulence: the fixed point and {O}nsager's conjecture.
\newblock {\em J. Math. Phys.}, 48(6):065503, 16, 2007.

\bibitem{cfp2010}
A.~Cheskidov, S.~Friedlander, and N.~Pavlovi\'{c}.
\newblock An inviscid dyadic model of turbulence: the global attractor.
\newblock {\em Discrete Contin. Dyn. Syst.}, 26(3):781--794, 2010.

\bibitem{cheskluosharpnonunique}
A.~Cheskidov and X.~Luo.
\newblock Sharp nonuniqueness for the {N}avier-{S}tokes equations.
\newblock {\em Invent. Math.}, 229(3):987--1054, 2022.

\bibitem{dai2023non}
M.~Dai and Q.~Peng.
\newblock Non-unique weak solutions of forced {SQG}.
\newblock {\em arXiv preprint arXiv:2310.13537}, 2023.

\bibitem{delellisszekelyhididifferentialinclusion}
C.~De~Lellis and L.~Sz\'{e}kelyhidi, Jr.
\newblock The {E}uler equations as a differential inclusion.
\newblock {\em Ann. of Math. (2)}, 170(3):1417--1436, 2009.

\bibitem{delellisszekelyhidicontinuousdissipative}
C.~De~Lellis and L.~Sz\'{e}kelyhidi, Jr.
\newblock Dissipative continuous {E}uler flows.
\newblock {\em Invent. Math.}, 193(2):377--407, 2013.

\bibitem{derosalerayhopf}
L.~De~Rosa.
\newblock Infinitely many {L}eray-{H}opf solutions for the fractional {N}avier-{S}tokes equations.
\newblock {\em Comm. Partial Differential Equations}, 44(4):335--365, 2019.

\bibitem{dn}
V.~Desnyansky and E.~Novikov.
\newblock The evolution of turbulence spectra to the similarity regime.
\newblock {\em Izv. Akad. Nauk SSSR Fiz. Atmos. Okeana}, 10(2):127--136, 1974.

\bibitem{elgindiII}
T.~M. Elgindi, T.-E. Ghoul, and N.~Masmoudi.
\newblock On the stability of self-similar blow-up for {$C^{1,\alpha}$} solutions to the incompressible {E}uler equations on {$\mathbb R^3$}.
\newblock {\em Camb. J. Math.}, 9(4):1035--1075, 2021.

\bibitem{filonovkhodunov}
N.~D. Filonov and P.~A. Khodunov.
\newblock Nonuniqueness of {L}eray--{H}opf solutions for a dyadic model.
\newblock {\em Algebra i Analiz}, 32(2):229--253, 2020.

\bibitem{fujitakato}
H.~Fujita and T.~Kato.
\newblock On the {N}avier-{S}tokes initial value problem. {I}.
\newblock {\em Arch. Rational Mech. Anal.}, 16:269--315, 1964.

\bibitem{guillodsverak}
J.~Guillod and V.~\v{S}ver\'{a}k.
\newblock Numerical investigations of non-uniqueness for the {N}avier-{S}tokes initial value problem in borderline spaces.
\newblock {\em J. Math. Fluid Mech.}, 25(3):Paper No. 46, 25, 2023.

\bibitem{hofmanova2023non}
M.~Hofmanov{\'a}, R.~Zhu, and X.~Zhu.
\newblock Non-uniqueness of {L}eray--{H}opf solutions for stochastic forced {N}avier--{S}tokes equations.
\newblock {\em arXiv preprint arXiv:2309.03668}, 2023.

\bibitem{iyer2020scaling}
K.~P. Iyer, K.~R. Sreenivasan, and P.~Yeung.
\newblock Scaling exponents saturate in three-dimensional isotropic turbulence.
\newblock {\em Physical Review Fluids}, 5(5):054605, 2020.

\bibitem{jiasverak}
H.~Jia and V.~Sverak.
\newblock Are the incompressible 3d {N}avier-{S}tokes equations locally ill-posed in the natural energy space?
\newblock {\em J. Funct. Anal.}, 268(12):3734--3766, 2015.

\bibitem{jiasverakinventiones}
H.~Jia and V.~\v{S}ver\'{a}k.
\newblock Local-in-space estimates near initial time for weak solutions of the {N}avier-{S}tokes equations and forward self-similar solutions.
\newblock {\em Invent. Math.}, 196(1):233--265, 2014.

\bibitem{katzpavlovicfinitetimeblowup}
N.~H. Katz and N.~Pavlovi\'{c}.
\newblock Finite time blow-up for a dyadic model of the {E}uler equations.
\newblock {\em Trans. Amer. Math. Soc.}, 357(2):695--708, 2005.

\bibitem{kiselevzlatosdiscrete}
A.~Kiselev and A.~Zlato\v{s}.
\newblock On discrete models of the {E}uler equation.
\newblock {\em Int. Math. Res. Not.}, 2005(38):2315--2339, 2005.

\bibitem{leray}
J.~Leray.
\newblock Sur le mouvement d'un liquide visqueux emplissant l'espace.
\newblock {\em Acta Math.}, 63(1):193--248, 1934.

\bibitem{obukhov}
A.~Obukhov.
\newblock Some general properties of equations describing the dynamics of the atmosphere.
\newblock {\em Academy of Sciences, USSR, Izvestiya, Atmospheric and Oceanic Physics}, 7:471--475, 1971.

\bibitem{taobook}
T.~Tao.
\newblock {\em Nonlinear dispersive equations}, volume 106 of {\em CBMS Regional Conference Series in Mathematics}.
\newblock Conference Board of the Mathematical Sciences, Washington, DC; by the American Mathematical Society, Providence, RI, 2006.
\newblock Local and global analysis.

\bibitem{taoaveraged}
T.~Tao.
\newblock Finite time blowup for an averaged three-dimensional {N}avier-{S}tokes equation.
\newblock {\em J. Amer. Math. Soc.}, 29(3):601--674, 2016.

\bibitem{taoblognavierstokesconvexintegration}
T.~Tao.
\newblock 255{B}, {N}otes 2: {O}nsager's {C}onjecture.
\newblock \url{https://terrytao.wordpress.com/2019/01/08/255b-notes-2-onsagers-conjecture}, January 2019.

\bibitem{vishik1}
M.~Vishik.
\newblock Instability and non-uniqueness in the {C}auchy problem for the {E}uler equations of an ideal incompressible fluid. {Part I}.
\newblock {\em arXiv preprint arXiv:1805.09426}, 2018.

\bibitem{vishik2}
M.~Vishik.
\newblock Instability and non-uniqueness in the {C}auchy problem for the {E}uler equations of an ideal incompressible fluid. {Part II}.
\newblock {\em arXiv preprint arXiv:1805.09440}, 2018.

\bibitem{wiedemanninviscidsymmetry}
E.~Wiedemann.
\newblock Inviscid symmetry breaking with non-increasing energy.
\newblock {\em C. R. Math. Acad. Sci. Paris}, 351(23-24):907--910, 2013.

\end{thebibliography}

\end{document}